\documentclass{article}

\usepackage[utf8]{inputenc}
\usepackage[cyr]{aeguill}
\usepackage{url}
\usepackage{graphicx}
\usepackage{bm}
\usepackage[a4paper]{geometry}
\usepackage{amsmath,amsfonts,amsthm,amssymb}
\usepackage{stmaryrd}
\usepackage{color,xcolor}
\usepackage{subcaption}
\usepackage{array}
\usepackage{multirow}
\usepackage{url}
\usepackage{hyperref}
\usepackage{centernot}

\newtheorem{example}{Example}
\newtheorem{remark}{Remark}
\newtheorem{proposition}{Proposition}
\newtheorem{corollary}{Corollary}
\newcommand{\R}{\mathbb R}

\renewcommand{\P}{\mathbb P}
\newcommand{\M}{\mathbb M}
\newcommand{\bbf}{\mathbf{f}}
\newcommand{\bbg}{\mathbf{g}}
\newcommand{\bbu}{\mathbf{u}}
\newcommand{\bbv}{\mathbf{v}}

\newcommand{\bbm}{\mathbf{m}}
\newcommand{\bF}{\mathbf{F}}
\newcommand{\dpar}[2]{\dfrac{\partial #1}{\partial #2}}

\DeclareSymbolFont{matha}{OML}{txmi}{m}{it}
\DeclareMathSymbol{\varv}{\mathord}{matha}{118}
\begin{document}
\title{A new local and explicit kinetic method for linear and non-linear convection-diffusion problems with finite kinetic speeds: \\ I. One-dimensional case}
\author{Gauthier Wissocq\footnote{corresponding author}, R\'emi Abgrall\\
Institute of Mathematics, University of Z\"urich, Switzerland\\
gauthier.wissocq@math.uzh.ch, remi.abgrall@math.uzh.ch}
\date{}
\maketitle
\begin{abstract}
We propose a numerical approach, of the BGK kinetic type, that is able to approximate with a given, but arbitrary, order of accuracy the solution of linear and non-linear convection-diffusion type problems: scalar advection-diffusion, non-linear scalar problems of this type and the compressible Navier-Stokes equations.
Our kinetic model can use \emph{finite} advection speeds that are independent of the relaxation parameter, and the time step does not suffer from a parabolic constraint. Having finite speeds is in contrast with many of the previous works about this kind of approach, and we explain why this is possible: paraphrasing more or less \cite{golse:hal-00859451}, the convection-diffusion like PDE is  not a limit of the BGK equation, but a correction of the same PDE without the parabolic term  at the second  order in the relaxation parameter that is interpreted as Knudsen number. We then show that introducing a matrix collision instead of the well-known BGK relaxation makes it possible to target a desired convection-diffusion system.

Several numerical examples, ranging from a simple pure diffusion model to the compressible Navier-Stokes equations illustrate our approach.
\end{abstract}

\section{Introduction}
We are interested in the approximation of linear and non-linear advection-diffusion equations using kinetic methods. Typically, this problem is addressed by considering models of the Jin-Xin type in the so-called diffusion limit~\cite{Jin, Jin1998, Jin2000, Naldi2000, Peng2020}. A representative example of such methods is expressed as follows:
\begin{align}
	& \frac{\partial u^\varepsilon}{\partial t} + \frac{\partial v^\varepsilon}{\partial x} = 0, \label{eq:u} \\
	& \frac{\partial v^\varepsilon}{\partial t} + \frac{1}{\varepsilon^2} \frac{\partial p(u^\varepsilon)}{\partial x} = \frac{1}{\varepsilon^2} (f(u^\varepsilon) - v^\varepsilon), \label{eq:v}
\end{align}
with $p'(u^\varepsilon) >0$ and where $\varepsilon$ is a smallness parameter referred to as the Knudsen number. Note that the above equations are written in dimensionless form, which justifies the fact that $\varepsilon$ has no dimension. In the diffusion limit as $\varepsilon \rightarrow 0$, the solution $u^\varepsilon$ formally converges to the solution of the following equation: 
\begin{align}\label{eq:v:para}
	\frac{\partial u}{\partial t} + \frac{\partial f(u)}{\partial x} = \frac{\partial^2 p(u)}{\partial x^2}.
\end{align}

Various approaches have been considered to  solve numerically this kinetic model in the literature. First, it is noteworthy that in the particular case of linear diffusion (where $p'(u^\varepsilon)$ is constant) diagonalizing the left-hand-side (transport) term of~(\ref{eq:u})-(\ref{eq:v}) allows us to write it as the following advection-relaxation system,
\begin{align}
	\frac{\partial}{\partial t} 
	\begin{bmatrix}
		f_1 \\ f_2
	\end{bmatrix}
	+ 
	\begin{bmatrix}
		-a & 0 \\ 0 & a
	\end{bmatrix}
	\frac{\partial}{\partial x}
	\begin{bmatrix}
		f_1 \\ f_2
	\end{bmatrix}
	=
	\frac{1}{\varepsilon^2}
	\begin{bmatrix}
		\mathbb{M}_1 - f_1 \\ \mathbb{M}_2 - f_2
	\end{bmatrix},
	\label{eq:intro_kinetic_model}
\end{align}
where $u^\varepsilon = f_1+f_2$, $v^\varepsilon = a(-f_1 + f_2)$, $a=\sqrt{p'}/\varepsilon$, $\mathbb{M}_1 = (u^\varepsilon - f(u^\varepsilon)/a)/2$ and $\mathbb{M}_2 = (u^\varepsilon + f(u^\varepsilon)/a)/2$. Note that $\mathbb{M}_1$ and $\mathbb{M}_2$ are commonly referred to as Maxwellian functions by analogy with the kinetic theory of gases. A first possibility is therefore to treat the left-hand-side term (advection at velocity $\pm a$) using an explicit scheme and the right-hand-side term (stiff relaxation) using an implicit scheme. An important problem of this approach is that the advection velocities scale as $1/\varepsilon$. As a consequence, the numerical stability constraint reads $\Delta t = \mathcal{O}(\varepsilon \Delta x)$, which is, in the diffusive limit where $\varepsilon < \Delta x$, more restrictive than the common parabolic constraint $\Delta t = \mathcal{O}(\Delta x^2)$~\cite{Jin1998, Boscarino2013, Peng2020}. 

To circumvent this issue, most previous work focused on the use of so-called partitioned schemes, where the stiff hyperbolic part is split into an explicit (non-stiff) term, and an implicit (stiff) term~\cite{Jin1998, Klar1998, Jin2000, Naldi2000, Jin2001, Aregba-Driollet2003, Lemou2008, LAFITTE20171, Jang2014, Peng2021}. Note that using a diagonally-implicit Runge-Kutta (DIRK) scheme for (\ref{eq:v}), the implicitness becomes linear and can be easily inverted since all the non-linear functions of $u^\varepsilon$ are known. In a sense, the fact of considering (a part of) the advection of $v^\varepsilon$ as a stiff term can be viewed as the introduction of space derivatives in the Maxwellian~\cite{Jin1998}. 
This consideration leads to two difficulties met by these approaches. The first one, reported in~\cite{Jin1998}, is the complexity of building stable implicit-explicit (IMEX) schemes for solving such systems. It is known that kinetic models in the form of~(\ref{eq:intro_kinetic_model}) are compatible with entropy inequalities when $\mathbb{M}_1$, $\mathbb{M}_2$ are monotone in the sense of~\cite{Bouchut}. In the  case of the standard Xin-Jin model, this condition is equivalent to Whitham subcharacteristic condition~\cite{Whitham1974, Jin}. However, when $\mathbb{M}_1$ and $\mathbb{M}_2$ depend on gradients, this property may be lost, which can explain a degraded robustness. The second problem, as shown in \cite{Boscarino2013}, is that such schemes suffer from a parabolic stability condition $\Delta t = \mathcal{O}(\Delta x^2)$. To solve this defect, the authors proposed a new partitioned model for the evolution of $u^\varepsilon$, which allowed them to successfully recover the hyperbolic CFL restriction $\Delta t = \mathcal{O}(\Delta x)$~\cite{Boscarino2013, Boscarino2014}. However, the drawback of this approach is the use of implicit methods to treat space gradients, which can be extremely costly in terms of computational time because large matrices have to be inverted~\cite{Bouchut2018}.
 
The main issue of the aforementioned approaches arises from the dependence of the characteristic velocities in $1/\varepsilon$. It is yet possible to consider another paradigm by noticing the way the Navier-Stokes equations can be derived from the Boltzmann equation in the kinetic theory of gases. With a Bhatnagar-Gross-Krook (BGK) collision operator~\cite{Bhatnagar1954}, the Boltzmann equation reads, in a dimensionless form~\cite{Boltzmann1872, golse:hal-00859451, Golse2021},
\begin{align}\label{eq:BGK}
	\frac{\partial f}{\partial t}(\boldsymbol{x}, \boldsymbol{\xi}, t) + \boldsymbol{\xi} \cdot \boldsymbol{\nabla}_{\boldsymbol{x}} f(\boldsymbol{x}, \boldsymbol{\xi}, t) = \frac{1}{\varepsilon} (f^{eq}(\boldsymbol{u}(\boldsymbol{x},t), \boldsymbol{\xi}) - f(\boldsymbol{x}, \boldsymbol{\xi}, t)),
\end{align} 
where $f:(\mathbb{R}^D \times \mathbb{R}^D \times \mathbb{R}^+) \mapsto \mathbb{R}^+$ is a referred to as a population related to the distribution of particles located at a position $\boldsymbol{x}$ in space, at time $t$ and moving with a microscopic velocity $\boldsymbol{\xi}$, $\varepsilon$ is the Knudsen number, $\boldsymbol{u}$ is the vector of conserved variables defined as
\begin{align}
	\boldsymbol{u}(\boldsymbol{x}, t) = \int_{\mathbb{R}^D} \left[ 1, \boldsymbol{\xi}, \frac{1}{2} ||\boldsymbol{\xi}||^2 \right]^T f(\boldsymbol{x}, \boldsymbol{\xi}, t) \, \mathrm{d}^D \boldsymbol{\xi},
\end{align}
and $f^{eq}$ is an equilibrium state usually considered as the Maxwell-Boltzmann distribution function~\cite{MAXWELL_PTRSL_157_1867}. It is paramount to notice that in the Boltzmann equation, only the collision term behaves as a stiff term, the advection velocities $\boldsymbol{\xi}$ being an additional variable of the system. Yet, it is possible to approximate the Boltzmann equation, at least formally, by the Navier-Stokes equations, including second-order diffusive terms. This is achieved by a introducing a first-order correction in $\varepsilon$ to the Euler equations, which is the purpose of the Chapman-Enskog expansion~\cite{Chapman1953}. On the contrary, all the aforementioned models based on the prototype (\ref{eq:u})-(\ref{eq:v}) target a desired PDE in the diffusion limit $\varepsilon \rightarrow 0$, which is very different. Interestingly, the hydrodynamic limits of the Boltzmann equation can be preserved by replacing the velocity space $\boldsymbol{\xi} \in \mathbb{R}^D$ by a finite set of discrete velocities $\boldsymbol{\xi}_i$, giving birth to the so-called discrete-velocity Boltzmann equations (DVBE)~\cite{Grad1949,Gatignol,Cabannes1,SHAN2006, Philippi2006}. The latter share many similarities with the diagonalized system (\ref{eq:intro_kinetic_model}), where $a$ has to be replaced by constant, arbitrary selected, discrete velocities, independent of $\varepsilon$. The fact that the advection velocities are constant makes it possible to build very simple numerical methods for solving the DVBE, which has notably made the great success of the lattice Boltzmann method (LBM), based on a simple collide and stream algorithm~\cite{Kruger2017}. The main issue of this method is its lack of numerical stability in the inviscid limit ($\varepsilon \rightarrow 0$) and for high-Mach compressible flows~\cite{Dellar2002, Siebert2008a, Wissocq2019, Coreixas2020}. This defect can be attributed to the fact that the DVBE is hardly compatible with entropy properties~\cite{Wagner1998}, even though many efforts have been devoted to recover a discrete counterpart of Boltzmann's H-theorem for the LBM~\cite{Karlin1998, Boghosian2001, Ansumali2003, Frapolli2015, ATIF_PRL_119_2017, Latt2020}. On the contrary, with a system \emph{\`a la} Xin-Jin, it is easy to find a Maxwellian that is compatible with a whole family of Lax entropies~\cite{Bouchut2018}.

The purpose of this paper is to introduce a new kinetic model for convection-diffusion problems that allows for hyperbolic stability conditions with $\Delta t = \mathcal{O}(\Delta x)$. This achievement is made possible by two innovative ideas. First, instead of targeting a desired PDE in the limit of a vanishing relaxation parameter, as is commonly done in the diffusion limit of kinetic systems, we want to recover the diffusive flux as the first-order term of an asymptotic expansion in a smallness parameter $\varepsilon$ referred to as the Knudsen number. This is strongly inspired by the way the Chapman-Enskog expansion is performed in kinetic theory. Secondly, we demonstrate that it is possible to control the diffusion of the $\mathcal{O}(\varepsilon)$-related terms to target a desired advection-diffusion system, with kinetic velocities that are independent of the Knudsen number. This involves modifying the BGK collision operator, in a way that is similar to the multiple relaxation times (MRT) that are well known in the LBM community~\cite{DHumieres1994, Lallemand2000, DHumieres2002}. By using an adequate time and space discretization, we show how it is possible to build robust numerical methods with Courant-Friedrichs-Lewy~\cite{Courant1967} (CFL) numbers close to unity without inverting large matrices in space. In this paper, we illustrate the methodology in the one-dimensional case. The extension to multi-dimensions, which requires additional considerations in the construction of the collision matrix, will be addressed in a forthcoming article.

It may seem counter-intuitive, and even in contradiction with previous works, to claim that we can construct methods with finite speeds of propagation while in previous works, special care has to be taken to overcome the issue of non bounded propagation speed. When considering \eqref{eq:v}, we look for method able to handle the limit case $\varepsilon \rightarrow 0$, because the problem \eqref{eq:v:para} is obtained in this limit. Hence one needs to be able to approximate correctly \eqref{eq:v} in this limit. In our work, we try to approximate the Chapman-Enskog expansion of \eqref{eq:BGK} (or more precisely a modification of it) for finite but non zero values of $\varepsilon$ in order to recover correctly the first terms of the development. The modification is constructed such that these first terms are exactly \eqref{eq:v:para}. The two approaches are very different.

The format of this paper is as follow. We begin by stating the problem and revisiting the hyperbolic models \emph{\`a la} Xi-Jin. Performing a Chapman-Enskog-like expansion, we observe that these models, at the leading order, resemble a parabolic equation with a very specific diffusive term. This leads us to propose a modification of the BGK relaxation term, in such a way that the true dissipative operator can be recovered for systems of equations. We show that this is always possible, modulo a standard sub-characteristic condition. This approach is applied to both scalar problems and systems. We explicitly construct the collision term for several wave models. Subsequently, we delve into the study of time discretization, employing a deferred correction IMEX method, and present some numerical results. Notably, we show that the correct entropy production is obtained for an exact solution of the Navier-Stokes equations.

\section{Problem statement}
\label{sec:problem_statement}

We are given the one-dimensional partial differential equation
\begin{align}
	    \frac{\partial \mathbf{u}}{\partial t} + \frac{\partial \mathbf{f}(\mathbf{u})}{\partial x} = \frac{\partial}{\partial x} \left( \mathbf{D} \frac{\partial \mathbf{u}}{\partial x} \right),\label{eq:transport_diffusion}
\end{align}
with $\mathbf{u} \in \mathbb{R}^p$, $\mathbf{f}: \mathbb{R}^p \rightarrow \mathbb{R}^p$ a Lipschitz continuous convective flux and $\mathbf{D}=\mathbf{D}(\mathbf{u})$ a $(p \times p)$ matrix which aims at introducing a diffusive flux in the transport equation. We assume that system \eqref{eq:transport_diffusion} can be written as a symmetric advective-diffusive system, meaning that there exists a strictly convex entropy $\eta = \eta(\mathbf{u})$ together with an entropy variable $\bbv = \nabla_{\bbu} \eta$ that symmetrizes it~\cite{Harten1983, mallet}. Eventually left-multiplying (\ref{eq:transport_diffusion}) by $\bbv^T$, where superscript $T$ denotes transpose, this reads
\begin{align}
    \dpar{\eta}{t} + \dpar{g(\bbu)}{x} = \bbv^T \dpar{}{x} \left( \mathbf{D} \dpar{\bbu}{x} \right),
\end{align}
where $g$ is the entropy flux defined by its gradient $(\nabla_{\bbu} g)^T = \bbv^T \nabla_{\bbu} \bbf$. Denoting $\mathbf{A}_0 = \bbv'(\bbu)$ the Hessian matrix of $\eta$ (which is positive definite thus invertible since $\eta$ is strictly convex) and assuming that $\mathbf{D} \mathbf{A}_0^{-1}$ is symmetric positive semi-definite~\cite{mallet}, we have
\begin{align}
   \bbv^T \dpar{}{x} \left( \mathbf{D} \dpar{\bbu}{x} \right) =
\dpar{}{x} \left( \bbv^T \mathbf{D} \dpar{\bbu}{x} \right) -\dpar{\bbv}{x}^T \mathbf{D} \mathbf{A}_0^{-1} \dpar{\bbv}{x} \leq \dpar{}{x} \left( \bbv^T \mathbf{D} \dpar{\bbu}{x} \right).
\end{align}
This leads to
\begin{align}
    \dpar{\eta}{t}+\dpar{\bbg(\bbu)}{x} - \dpar{}{x} \left( \bbv^T \mathbf{D} \dpar{\bbu}{x} \right) \leq 0,
\end{align}
which, when applied to the Navier-Stokes system of equations for gas dynamics, leads to the Clausius-Duhem inequality~\cite{mallet}. A last remark is that, since 
\begin{align}
    \mathbf{D} = \mathbf{A}_0^{-1/2} \left( \mathbf{A}_0^{1/2} \mathbf{D} \mathbf{A}_0^{-1} \mathbf{A}_0^{1/2} \right) \mathbf{A}_0^{1/2},
\end{align}
then $\mathbf{D}$ is similar to the symmetric positive semi-definite matrix $\mathbf{A}_0^{1/2} \mathbf{D} \mathbf{A}_0^{-1} \mathbf{A}_0^{1/2}$, sharing all its eigenvalues. Hence, $\mathbf{D}$ has real non-negative eigenvalues.

Compared to numerical methods for hyperbolic systems obeying a stability condition $\Delta t=\mathcal{O}(\Delta x)$, the presence of second-order derivatives in the diffusion term of (\ref{eq:transport_diffusion}) introduces a parabolic stability constraint $\Delta t = \mathcal{O}(\Delta x^2)$ when it is explicitly solved. To overcome this limitation, we want to deal with a kinetic model involving first-order derivatives only, with arbitrarily fixed velocities, and accounting for diffusion through a purely local relaxation term. We first recall the kinetic model adopted in \cite{Jin,AregbaNatalini} and subsequently in \cite{Torlo} to solve the PDE (\ref{eq:transport_diffusion}) when $\mathbf{D} = \mathbf{0}$.

\subsection{Kinetic model for hyperbolic equations}

In \cite{Jin,AregbaNatalini}, the following BGK model is considered to solve (\ref{eq:transport_diffusion}) with $\mathbf{D}=\mathbf{0}$ (hyperbolic transport equation):
\begin{align}
    \frac{\partial \mathbf{F}}{\partial t} + \Lambda \frac{\partial \mathbf{F}}{\partial x} = \frac{\mathbb{M}(\mathbb{P} \mathbf{F}) - \mathbf{F}}{\tau},
    \label{eq:kinetic_equation}
\end{align}
where $\mathbf{F} \in \mathbb{R}^{kp}$, $k$ is the number of waves of the kinetic model, $\Lambda$ is a diagonal matrix, constant in space and time, $\mathbb{M}$ plays the role of a Maxwellian,  $\mathbb{P}$ is a linear operator such that $\mathbb{P}\mathbb{M} (\mathbb{P} \mathbf{F}) = \mathbb{P} \mathbf{F}$. The parameter $\tau$ plays the role of a the relaxation time. To study this kinetic system, it is important to understand how we can introduce a Knudsen number. We do it here drawing inspiration from the kinetic theory of gases, which can for example be found in~\cite{golse:hal-00859451, Golse2021}. Looking at \eqref{eq:kinetic_equation} we see that if we multiply $\bF$ by some factor, provided that the Maxwellian is homogeneous of degree 1 in $\bF$, nothing changes. All the models of Maxwellians satisfy this property. Choose now a characteristic length of the problem under consideration $\ell$ and a characteristic time $\theta$. We define dimensionless time, space and velocity matrix as
\begin{align}
	t^* = \frac{t}{\theta}, \qquad x^* = \frac{x}{\ell}, \qquad \Lambda^* = \frac{\Lambda}{||\Lambda||},
\end{align}
where $||\Lambda||$ is the $L^2$ norm of the diagonal matrix $\Lambda$, i.e. the maximum of the absolute values of the diagonal entries. Eq.~\eqref{eq:kinetic_equation} becomes
\begin{align}
	\frac{1}{\theta} \dpar{\bF}{t^*} + \frac{||\Lambda||}{\ell} \Lambda^* \dpar{\bF}{x^*} = \frac{\mathbb{M}(\mathbb{P} \mathbf{F}) - \mathbf{F}}{\tau}.
\end{align}
If we want to solve the same problem, we need to set $\theta=\ell/||\Lambda||$ and define a Knudsen number $\varepsilon$ as
\begin{align}
	\varepsilon \equiv \frac{||\Lambda|| \tau}{\ell}, 
\end{align}
so that the dimensionless form of \eqref{eq:kinetic_equation} reads
\begin{equation}
    \frac{\partial \mathbf{F}}{\partial t^*} + \Lambda^* \frac{\partial \mathbf{F}}{\partial x^*} = \frac{\mathbb{M}(\mathbb{P} \mathbf{F}) - \mathbf{F}}{\varepsilon}.
    \label{eq:kinetic_equation_adim}
\end{equation}
Doing this scaling, we see that we can compare $\big ( \mathbb{M}(\mathbb{P} \mathbf{F})-\bF)$ and 
$\Lambda^* \partial \bF/\partial x^*$ because they have the same dimensions. This is notably the purpose of the Chapman-Enskog expansion. When $\varepsilon \ll 1$, $\bF$ remains close to $\M(\P \bF)$, while when $\varepsilon \centernot{\ll} 1$, perturbations about the Maxwellian have to be considered.

Another form of \eqref{eq:kinetic_equation}, maybe less familiar, is
\begin{align}
	\dpar{\bF}{t} + \Lambda \dpar{\bF}{x} = \frac{||\Lambda||}{\ell} \frac{\mathbb{M}(\mathbb{P} \mathbf{F}) - \mathbf{F}}{\varepsilon},
	\label{eq:kinetic_equation_retained}
\end{align}
which is the form of BGK system we adopt in the rest of this section.

\begin{remark}
This kind of consideration never appears in above mentioned references 
   because these authors want to work in the limit $\varepsilon\rightarrow 0$, i.e. $\tau\rightarrow 0$. In our case, we need to work in the case of a finite but small $\varepsilon$. By itself, small is meaningless. Small is small with respect to something else only. This is the reason why we need to define $\varepsilon$.
\end{remark}

It can be shown that when the transport matrix $\Lambda$ and the Maxwellian $\mathbb{M}$ are related to the convective flux $\mathbf{f}$ as $\mathbb{P} \Lambda \mathbb{M} (\mathbb{P} \mathbf{F}) = \mathbf{f}(\mathbb{P} \mathbf{F})$, then the hyperbolic system (\ref{eq:transport_diffusion}) 
 with $\mathbf{D}=\mathbf{0}$ is the formal limit of (\ref{eq:kinetic_equation_retained})
  when $\varepsilon \rightarrow 0$, with $\mathbf{u}=\mathbb{P} \mathbf{F}$. Following \cite{Bouchut}, the choice of $\Lambda$ is made such that the eigenvalues of $\mathbb{M}$ with respect to $\mathbf{u}$ are in $\mathbb{R}^+$: this fundamental property ensures the existence of an entropy for the kinetic system.

\begin{example}[scalar conservation equation]
\label{ex:scalar}
    The simplest example is a two-wave model ($k=2$) for solving a scalar conservation equation ($p=1$). We take
\begin{align}
    \mathbf{F} = 
    \begin{pmatrix}
        f_1 \\ f_2
    \end{pmatrix}, \qquad
    \Lambda =
    \begin{pmatrix}
        -a & 0 \\ 0 & a
    \end{pmatrix}, \qquad
    \mathbb{P} = 
    \begin{pmatrix}
        1 & 1
    \end{pmatrix} \qquad \mathrm{and} \qquad
    \mathbb{M}= 
    \begin{pmatrix}
        \mathbb{M}_1 \\ \mathbb{M}_2
    \end{pmatrix},
\end{align}
with $a>0$ and with
\begin{align}
    \begin{cases}
    	\mathbb{P} \mathbb{M} = \mathbb{P} \mathbf{F} \equiv u^\varepsilon, \\
    	\mathbb{P} \Lambda \mathbb{M} =f(u^\varepsilon),
    \end{cases}
    \Rightarrow
    \begin{cases}
    	\mathbb{M}_1+\mathbb{M}_2=f_1 + f_2 \equiv u^\varepsilon, \\
    	a(-\mathbb{M}_1 + \mathbb{M}_2) = f(u^\varepsilon).
    \end{cases}
\end{align}
These two conditions are sufficient to construct a Maxwellian. The system (\ref{eq:intro_kinetic_model}) is recovered with constant kinetic speeds, independent of the relaxation parameter. We also know that when $a$ is chosen such that $|f'(u^\varepsilon)| <a $ (subcharacteristic condition), the two-wave model becomes compatible with entropy inequalities~\cite{Jin, Bouchut}.
\end{example}

\begin{example}[Euler equations for fluid dynamics] \label{ex:Euler}
    Another example is a two-wave model ($k=2$) for the 1D Euler equations for fluid dynamics, ensuring the conservation of mass $\rho$, momentum $j$ and energy $E$ ($p=3$). We define
\begin{align}
    \mathbf{F} = 
    \begin{pmatrix}
        \rho_1 \\ j_1 \\ E_1 \\ \rho_2 \\ j_2 \\Â E_2
    \end{pmatrix}, \qquad \Lambda = 
    \begin{pmatrix}
        -a & 0 & 0 & 0 & 0 & 0 \\
        0 & -a & 0 & 0 & 0 & 0 \\
        0 & 0 & -a & 0 & 0 & 0 \\
        0 & 0 & 0 & a & 0 & 0 \\
        0 & 0 & 0 & 0 & a & 0 \\
        0 & 0 & 0 & 0 & 0 & a
    \end{pmatrix}, \qquad \mathbb{P} = 
    \begin{pmatrix}
        1 & 0 & 0 & 1 & 0 & 0 \\
        0 & 1 & 0 & 0 & 1 & 0 \\
        0 & 0 & 1 & 0 & 0 & 1
    \end{pmatrix}, \qquad 
    \mathbb{M} = 
    \begin{pmatrix}
        \mathbb{M}_1^\rho \\ \mathbb{M}_1^j \\ \mathbb{M}_1^E \\ \mathbb{M}_2^\rho \\ \mathbb{M}_2^j \\ \mathbb{M}_2^E
    \end{pmatrix},
\end{align}
with $a>0$ and with
\begin{align}
\begin{cases}
	\mathbb{P} \mathbb{M}= \mathbb{P} \mathbf{F} \equiv \mathbf{u}^\varepsilon, \\ \newline \\
	\mathbb{P} \Lambda \mathbb{M}=\mathbf{f}(\mathbf{u}^\varepsilon),
\end{cases}
\Rightarrow
\begin{cases}
	\begin{pmatrix}
        \mathbb{M}_1^\rho + \mathbb{M}_2^\rho \\
        \mathbb{M}_1^j + \mathbb{M}_2^j \\
        \mathbb{M}_1^E + \mathbb{M}_2^E
    \end{pmatrix} = \begin{pmatrix}
        \rho_1 + \rho_2 \\
        j_1 + j_2 \\
        E_1 + E_2 
    \end{pmatrix} \equiv
    \begin{pmatrix}
        \rho \\ j \\ E
    \end{pmatrix}, \\ \newline \\
    a \begin{pmatrix}
        - \mathbb{M}_1^\rho + \mathbb{M}_2^\rho \\
        - \mathbb{M}_1^j + \mathbb{M}_2^j \\
        - \mathbb{M}_1^E + \mathbb{M}_2^E
    \end{pmatrix} =
    \begin{pmatrix}
        j \\ j^2/\rho + P \\ (E + P)j/\rho
    \end{pmatrix},
\end{cases} \\
     \label{eq:conditions_Maxwellian_Euler}
\end{align}
where $P$, the thermodynamic pressure, is related to $\mathbf{u}^\varepsilon$ by an appropriate equation of state. This system of equations is always invertible, so that we can find a Maxwellian state satisfying conditions 
(\ref{eq:conditions_Maxwellian_Euler}). When $\rho(\bbf'(\bbu^\varepsilon)) <a$, where $\rho(\mathbf{M})$ denotes the spectral radius of a matrix $\mathbf{M}$, this model becomes compatible with entropy inequalities \cite{Bouchut}.
\end{example}

The questions of the present work are: can we approach a transport equation including a diffusive flux $-\mathbf{D} \partial_x \mathbf{u}$ for ``small'' values of $\varepsilon$ with a kinetic system such as (\ref{eq:kinetic_equation_retained})?
 Can we build explicit high-order numerical schemes based on the idea of~\cite{Torlo} to solve such transport-diffusion problems? It is noteworthy that we want to preserve the essential properties of the method developed in \cite{Torlo}, which are:
\begin{enumerate}
	\item[(a)] the scheme is computationally explicit involving local matrices (in space) only,
	\item[(b)] it is stable with hyperbolic stability conditions $\Delta t = \mathcal{O}(\Delta x)$ for CFL numbers close to or even above $1$,
	\item[(c)] the convergence order in time and space can be arbitrarily chosen.
\end{enumerate}
In particular, concerning (a), we aim to avoid the need to invert large matrices involving multiple spatial points for the sake of efficiency and memory purposes. This is why we refrain from using implicit time integration schemes to to handle space derivatives, as proposed in previous work~\cite{Boscarino2013, Boscarino2014}.

\subsection{First attempt based on the Chapman-Enskog expansion}
\label{sec:naive_CE}

In this section, we first perform a Chapman-Enskog expansion of (\ref{eq:kinetic_equation_retained}) 
to show that the BGK kinetic model may not be appropriate to approximate the advection-diffusion problem (\ref{eq:transport_diffusion}),
 and this will help to suggest a solution. Note that, although the mathematical rigor of the Chapman-Enskog expansion may be open to question\footnote{It should be noted that in some cases, this expansion can lead to non-physical and unstable macroscopic equations at the third-order, such as the Burnett equations, as observed in the context of the kinetic theory of gases~\cite{Bobylev2006}.}, we apply it here in the construction of kinetic systems for numerical schemes which will be validated \textit{a posteriori}.

First, note that \eqref{eq:kinetic_equation_retained} is equivalent to

\begin{equation}\label{eq:F_M_epsilon2}
	\mathbf{F} = \mathbb{M}(\mathbb{P} \mathbf{F}) - \varepsilon \omega \left[ \frac{\partial \mathbf{F}}{\partial t} + \Lambda \frac{\partial \mathbf{F}}{\partial x} \right],
\end{equation}
where we define
\begin{align}
	\omega = \frac{\ell}{||\Lambda||}.
\end{align}
Looking at \eqref{eq:F_M_epsilon2}, we see that different regimes may be considered depending on the value of $\varepsilon$. When $\varepsilon \ll 1$, the effects of collisions dominate and distribution functions $\bF$ are very close to the Maxwellian state $\M(\P \bF)$. This reads
\begin{align}
    \mathbf{F} = \mathbb{M}(\mathbb{P} \mathbf{F}) + \mathcal{O}(\varepsilon).
    \label{eq:F_M_epsilon}
\end{align}
Injecting (\ref{eq:F_M_epsilon}) in (\ref{eq:F_M_epsilon2}) yields an approximation of $\mathbf{F}$ up to the second-order in $\varepsilon$:
\begin{align}
	\mathbf{F} = \mathbb{M}(\mathbf{u}^\varepsilon) - \varepsilon \omega \left[ \frac{\partial \mathbb{M}(\mathbf{u}^\varepsilon)}{\partial t} + \Lambda \frac{\partial \mathbb{M}(\mathbf{u}^\varepsilon)}{\partial x} \right] + \mathcal{O}(\varepsilon^2),
\end{align}
where $\mathbf{u}^\varepsilon = \mathbb{P}\mathbf{F}$. Then left multiplying by the constant matrix $\mathbb{P} \Lambda$,
\begin{align}
	\mathbb{P} \Lambda \mathbf{F} = \mathbf{f}(\mathbf{u}^\varepsilon) - \varepsilon \omega \left[ \frac{\partial \mathbf{f}(\mathbf{u}^\varepsilon)}{\partial t} + \frac{\partial \mathbb{P} \Lambda^2 \mathbb{M}(\mathbf{u}^\varepsilon)}{\partial x} \right] + \mathcal{O}(\varepsilon^2), \label{eq:P_Lam_F_1}
\end{align}
where we used the fact that $\mathbb{P} \Lambda \mathbb{M} (\mathbf{u}^\varepsilon) = \mathbf{f}(\mathbf{u}^\varepsilon)$ by construction of the Maxwellian. In the kinetic theory gases, the quantity $\mathbb{P} \Lambda^2 \mathbb{M}(\mathbf{u}^\varepsilon) \equiv \mathbf{m}_2(\mathbf{u}^\varepsilon)$ is commonly referred to as the second-order moment of the Maxwellian $\mathbb{M}$. The time-derivative of $\mathbf{f}(\mathbf{u}^\varepsilon)$ can be addressed using a chain rule as
\begin{align}
	\frac{\partial \mathbf{f} (\mathbf{u}^\varepsilon)}{\partial t} = \mathbf{f}'(\mathbf{u}^\varepsilon) \frac{\partial \mathbf{u}^\varepsilon}{\partial t}.
\end{align}
Then, applying the projector $\mathbb{P}$ on (\ref{eq:kinetic_equation_retained}) yields
\begin{align}
\dpar{\bbu^\varepsilon}{t}+\dpar{\P\Lambda\bF}{x}=0, \label{eq:P_kinetic_equation}
\end{align}
so that
\begin{align}
	\dpar{ \bbu^\varepsilon}{t} = -\dpar{\P \Lambda \bF}{x} = -\dpar{\mathbf{f}(\bbu^\varepsilon)}{x} + \mathcal{O}(\varepsilon) = -\mathbf{f}'(\bbu^\varepsilon) \dpar{\bbu^\varepsilon}{x} +  \mathcal{O}(\varepsilon).
\end{align}
Hence,
\begin{align}
	\dpar{\mathbf{f}(\bbu^\varepsilon)}{t} = - \left( \mathbf{f}'(\bbu^\varepsilon) \right)^2 \dpar{\bbu^\varepsilon}{x} + \mathcal{O}(\varepsilon).
\end{align}
Furthermore, using a chain rule,
\begin{align}
	\frac{\partial \mathbb{P} \Lambda^2 \mathbb{M}(\mathbf{u}^\varepsilon)}{\partial x} = \frac{\partial \mathbf{m}_2 (\mathbf{u}^\varepsilon)}{\partial x} = \mathbf{m}_2'(\mathbf{u}^\varepsilon) \frac{\partial \mathbf{u}^\varepsilon}{\partial x},
\end{align}
so that Eq.~(\ref{eq:P_Lam_F_1}) yields
\begin{align}
	\mathbb{P} \Lambda \mathbf{F} = \mathbf{f}(\mathbf{u}^\varepsilon) - \varepsilon \omega \left[ \mathbf{m}_2'(\mathbf{u}^\varepsilon) - \left( \mathbf{f}'(\mathbf{u}^\varepsilon) \right)^2 \right] \frac{\partial \mathbf{u}^\varepsilon}{\partial x} + \mathcal{O}(\varepsilon^2).
\end{align}
Using this approximation for the transport term of (\ref{eq:P_kinetic_equation}) results in
\begin{align}
	\frac{\partial \mathbf{u}^\varepsilon}{\partial t} + \frac{\partial \mathbf{f}(\mathbf{u}^\varepsilon)}{\partial x} = \varepsilon \omega \frac{\partial}{\partial x} \left\{ \left[ \mathbf{m}_2'(\mathbf{u}^\varepsilon) - \left( \mathbf{f}'(\mathbf{u}^\varepsilon) \right)^2 \right] \frac{\partial \mathbf{u}^\varepsilon}{\partial x} \right\} + \mathcal{O}(\varepsilon^2).
\end{align}
Recall that $\varepsilon \omega = \tau$. This is an approximation of (\ref{eq:transport_diffusion}) up to the second-order in $\varepsilon$ if we can ensure that
\begin{align}
	\tau \left[ \mathbf{m}_2'(\mathbf{u}^\varepsilon) - \left( \mathbf{f}'(\mathbf{u}^\varepsilon) \right)^2 \right] \stackrel{!}{=} \mathbf{D}(\mathbf{u}^\varepsilon).
 \label{eq:condition_with_BGK}
\end{align}
Various interpretations can be given to this equation:
\begin{enumerate}
    \item For given $\tau$ and Maxwellian $\M$ (thus having $\mathbf{m}_2'(\mathbf{u}^\varepsilon)$ determined), Eq.~(\ref{eq:condition_with_BGK}) exhibits the diffusive behavior $\mathbf{D}$ of the asymptotic system on $\mathbf{u}^\varepsilon$ at first-order in $\varepsilon$. Notably, when a two-wave model with velocities $(-a, a)$ is considered, note that $\Lambda^2=a^2 \mathbf{I}_{kp}$, where $\mathbf{I}_{kp}$ is the $(kp \times kp)$ identity matrix. Therefore, $\bbm_2'(\bbu^\varepsilon) = a^2\mathbf{I}_p$ and a key implication of the subcharacteristic condition is recovered in (\ref{eq:condition_with_BGK}): $\mathbf{D}$ has positive eigenvalues if and only if $ \rho(\mathbf{f}'(\mathbf{u})) < a$.
    \item With a given $\tau$ and a prescribed diffusion matrix $\mathbf{D}$, Eq.~(\ref{eq:condition_with_BGK}) can be seen as a requirement on $\M$ (\textit{via} $\mathbf{m}_2'$) to approximate (\ref{eq:transport_diffusion}). However, this strategy is impractical for constructing a numerical scheme for (\ref{eq:transport_diffusion}) for two reasons: (i) when a system of equations is considered ($p \geq 2$), the condition on $\mathbf{m}_2'$ cannot, in general, be integrated to find a Maxwellian $\M$ satisfying it ; (ii) even when this condition can be integrated, the resulting Maxwellian $\M$ may not adhere to the convexity properties required in~\cite{Bouchut} to fulfill the entropy inequalities essential for ensuring the stability of the model.
    \item For a given Maxwellian $\M$ and a prescribed diffusion matrix $\mathbf{D}$, Eq.~(\ref{eq:condition_with_BGK}) provides a condition on $\tau$ to approach the target equation (\ref{eq:transport_diffusion}) when a scalar system is considered ($p=1$). The approximation is then reasonable as far as $\varepsilon \ll 1$. As shown in the numerical validation of Sec.~\ref{sec:application_scalars}, this strategy can be adopted for scalar cases, and the Knudsen number can be arbitrarily reduced by modifying the kinetic velocities in $\Lambda$. However, this strategy is not directly applicable when dealing with a system of equations ($p \geq 2$).
\end{enumerate}

\begin{example}[scalar conservation equation with a two-wave model ($p=1$, $k=2$)]
\label{ex:tau_scalar}
Applying this last strategy to the model given in Example~\ref{ex:scalar}, Eq.~(\ref{eq:condition_with_BGK}) yields
\begin{align}
    \tau = \frac{D(u^\varepsilon)}{a^2 - (f'(u^\varepsilon))^2},
\end{align}
where $D(u^\varepsilon)$ is a positive diffusion 
coefficient. The strict sub-characteristic condition ensures that $\tau >0$.
\end{example}

The main objective of the next section is to introduce new relaxation models based on a collision matrix, with the intention of extending the observation made in the third point above to systems of equations. 

\section{Collision matrix approach}
\label{sec:multi-relaxation-approach}

The choice of a Maxwellian has a significant impact on the numerical stability of a kinetic scheme. As shown in~\cite{Bouchut}, when it adheres to a monotonicity condition, the BGK model is compatible with entropy inequalities. Consequently, our motivation is to keep the same Maxwellian as in previous work~\cite{Natalini, Torlo} to preserve these paramount properties. The introduction of new free parameters necessary to approximate the diffusion term $\mathbf{D}$ is accomplished by introducing a collision matrix in place of the BGK model. 

Using the same characteristic length and velocity as in \eqref{eq:kinetic_equation_retained}, the adopted collision matrix model reads
\begin{align}
	\frac{\partial \bF}{\partial t} + \Lambda \frac{\partial \bF}{\partial x} = \dfrac{\Vert \Lambda\Vert}{\ell}\frac{\Omega}{\varepsilon} (\mathbb{M}(\mathbb{P} \mathbf{F}) - \mathbf{F}). \label{eq:kinetic_MRT}
\end{align}
This kinetic model is similar to \eqref{eq:kinetic_equation_retained} except that a square matrix $\Omega \in \mathcal{M}_{kp}(\mathbb{R})$, to be defined, has been introduced. The Knudsen number $\varepsilon$ is also a quantity to be defined. As in the BGK model of Sec.~\ref{sec:problem_statement}, the characteristic length and kinetic velocity appear because $\Lambda$ is a free parameter, and we need to quantify the ratio between the ``collision'' terms and the ``advection'' ones to perform a Chapman-Enskog expansion. We will look for $\Omega$ such that $\Vert \Omega\Vert=O(1)$.

Our question is now the following: for a given Maxwellian $\mathbb{M}$, can we define a collision matrix $\Omega$ so that (\ref{eq:kinetic_MRT}) approaches (\ref{eq:transport_diffusion}) for arbitrarily small values of $\varepsilon$?

\subsection{Conservation condition on the collision matrix}

A first condition on $\Omega$ is to satisfy 
\begin{align}
    \mathbb{P} \Omega (\mathbb{M}(\mathbb{P} \mathbf{F}) - \mathbf{F}) = \mathbf{0},
\end{align}
ensuring the conservation of the quantity $\mathbf{u}^\varepsilon = \P \bF$. Let us build a general matrix $\Omega$ satisfying this condition. To fix the ideas and without loss of generality, we will adopt the conventions adopted in Example~\ref{ex:Euler} to define the components of $\mathbf{F}$. This means, the first $p$ lines of $\mathbf{F}$ are associated to the first wave of the model, and so on (in general: lines between $(i-1)p+1$ and $ip$ are associated to the wave $i$). We also assume that $\mathbb{P}$ has the block-matrix shape
\begin{align}
    \mathbb{P} = 
    \begin{pmatrix}
        \mathbf{I}_p & \hdots & \mathbf{I}_p
    \end{pmatrix}.
\end{align}
Then we choose $\Omega$ as an identity block matrix,
\begin{align}
    \Omega = 
    \begin{pmatrix}
        \tilde{\Omega} & \hdots & 0 \\
        \vdots & \ddots & \vdots \\
        0 & \hdots & \tilde{\Omega}
    \end{pmatrix} = \mathbf{I}_k \otimes \tilde{\Omega} \in \mathcal{M}_{kp} (\R), \label{eq:Collision_matrix_block}
\end{align}
where $\tilde{\Omega} \in \mathcal{M}_{p} (\R)$ is a matrix to be defined, $\mathbf{I}_k$ is the $(k \times k)$ identity matrix and symbol $\otimes$ stands for the Kronecker product of two matrices. The latter is defined for two matrices $\mathbf{A}=(a_{ij})\in \mathcal{M}_k(\R)$ and $\mathbf{B}\in \mathcal{M}_r(\R)$ as
\begin{align}
    \mathbf{A}\otimes \mathbf{B}=\begin{pmatrix}
a_{11}\mathbf{B} & \ldots & a_{1n}\mathbf{B}\\
\vdots  & \ddots &\vdots\\
a_{k1}\mathbf{B}  &\ldots &a_{kk}\mathbf{B}\end{pmatrix} \in \mathcal{M}_{kr}(\R).
\end{align}
The adopted form of $\Omega$ means that we assume a similar relaxation parameter for all the distributions carrying a given variable of $\mathbf{u}^\varepsilon$.

\begin{example}[Conservation equations for fluid dynamics]
    Adopting the notations of Example \ref{ex:Euler}, the multi-relaxation model yields the following PDE:
    \begin{align}
        & \frac{\partial}{\partial t}
        \begin{pmatrix}
            \rho_1 \\ j_1 \\ E_1
        \end{pmatrix}
        - \frac{\partial}{\partial x} \begin{pmatrix}
            \rho_1 \\ j_1 \\ E_1
        \end{pmatrix}
         = \frac{\tilde{\Omega}}{\omega \varepsilon}
         \begin{pmatrix}
             \M_1^\rho - \rho_1 \\
             \M_1^j - j_1 \\
             \M_1^E - E_1
         \end{pmatrix}, \\
         & \frac{\partial}{\partial t}
        \begin{pmatrix}
            \rho_2 \\ j_2 \\ E_2
        \end{pmatrix}
        + \frac{\partial}{\partial x} \begin{pmatrix}
            \rho_2 \\ j_2 \\ E_2
        \end{pmatrix}
         = \frac{\tilde{\Omega}}{\omega \varepsilon}
         \begin{pmatrix}
             \M_2^\rho - \rho_2 \\
             \M_2^j - j_2 \\
             \M_2^E - E_2
         \end{pmatrix}.
    \end{align}
\end{example}

With the choice of Eq.~(\ref{eq:Collision_matrix_block}), it is clear that $\mathbb{P} \Omega = \tilde{\Omega} \mathbb{P}$, so that
\begin{align}
    \mathbb{P} \Omega (\mathbb{M}(\mathbb{P} \mathbf{F}) - \mathbf{F}) = \tilde{\Omega} \mathbb{P} (\mathbb{M}(\mathbb{P} \mathbf{F}) - \mathbf{F}) = \mathbf{0},
\end{align}
and
\begin{align}
    \frac{\partial \mathbf{u}^\varepsilon}{\partial t} + \frac{\partial (\mathbb{P} \Lambda \mathbf{F})}{\partial x} = 0, 
\end{align}
meaning that the components of $\mathbf{u}^\varepsilon$ are conserved by construction. 

\subsection{Chapman-Enskog expansion}\label{sec:ChapmanE}

Let us now perform a similar expansion as in Sec.~\ref{sec:naive_CE} to obtain an approximation of the flux term $\mathbb{P} \Lambda \mathbf{F}$.   Eq.~(\ref{eq:kinetic_MRT}) yields
\begin{align}
    \mathbf{F} &= \mathbb{M}(\mathbf{u}^\varepsilon) - \varepsilon \omega\Omega^{-1} \left[ \frac{\partial \mathbf{F}}{\partial t} + \Lambda \frac{\partial \mathbf{F}}{\partial x} \right] \nonumber \\
    & = \mathbb{M}(\mathbf{u}^\varepsilon) - \varepsilon \omega \Omega^{-1} \left[ \frac{\partial \mathbb{M}(\mathbf{u}^\varepsilon)}{\partial t} + \Lambda \frac{\partial \mathbb{M}(\mathbf{u}^\varepsilon)}{\partial x} \right] + \mathcal{O}(\varepsilon^2).
\end{align}
Then left-multiplying by $\mathbb{P} \Lambda$:
\begin{align}
    \mathbb{P} \Lambda \mathbf{F} = \mathbf{f}(\mathbf{u}^\varepsilon) - \varepsilon \omega \mathbb{P} \Lambda \Omega^{-1}  \left[ \frac{\partial \mathbb{M}(\mathbf{u}^\varepsilon)}{\partial t} + \Lambda \frac{\partial \mathbb{M}(\mathbf{u}^\varepsilon)}{\partial x} \right] + \mathcal{O}(\varepsilon^2).
\end{align}
Given the block-matrix shapes of $\mathbb{P}$, $\Lambda$ and $\Omega$, we have $\mathbb{P} \Lambda \Omega^{-1} = \tilde{\Omega}^{-1}\mathbb{P} \Lambda$, so that
\begin{align}
    \mathbb{P}\Lambda \mathbf{F} = \mathbf{f}(\mathbf{u}^\varepsilon) - \varepsilon \omega \tilde{\Omega}^{-1} \left[ \frac{\partial \mathbf{f} (\mathbf{u}^\varepsilon)}{\partial t} + \frac{\partial \mathbf{m}_2 (\mathbf{u}^\varepsilon)}{\partial x} \right] + \mathcal{O}(\varepsilon^2).
\end{align}
Using similar chain rules as in Sec.~\ref{sec:naive_CE} gives
\begin{align}
    \mathbb{P} \Lambda \mathbf{F} = \mathbf{f}(\mathbf{u}^\varepsilon) - \varepsilon \omega \tilde{\Omega}^{-1} \left[ \mathbf{m}_2'(\mathbf{u}^\varepsilon) - \left( \mathbf{f}'(\mathbf{u}^\varepsilon) \right)^2 \right] \frac{\partial \mathbf{u}^\varepsilon}{\partial x} + \mathcal{O}(\varepsilon^2),
\end{align}
so that the following conservation equation can be obtained:
\begin{align}
	\frac{\partial \mathbf{u}^\varepsilon}{\partial t} + \frac{\partial \mathbf{f}(\mathbf{u}^\varepsilon)}{\partial x} = \varepsilon \frac{\partial}{\partial x} \left\{ \omega \tilde{\Omega}^{-1} \left[ \mathbf{m}_2'(\mathbf{u}^\varepsilon) - \left( \mathbf{f}'(\mathbf{u}^\varepsilon) \right)^2 \right] \frac{\partial \mathbf{u}^\varepsilon}{\partial x} \right\} + \mathcal{O}(\varepsilon^2). \label{eq:CE_transport_diffusion}
\end{align}
This is an approximation of Eq.~(\ref{eq:transport_diffusion}) up to the first-order in $\varepsilon$ if we can ensure that
\begin{align}
   \varepsilon \omega \tilde{\Omega}^{-1} \left[ \mathbf{m}_2'(\mathbf{u}^\varepsilon) - \left( \mathbf{f}'(\mathbf{u}^\varepsilon) \right)^2 \right] = \mathbf{D}(\mathbf{u}^\varepsilon). \label{eq:condition_with_MRT}
\end{align}
Assuming that $\left[ \mathbf{m}_2'(\mathbf{u}^\varepsilon) - \left( \mathbf{f}'(\mathbf{u}^\varepsilon) \right)^2 \right]$ is invertible, this yields a relationship satisfied by $ \varepsilon \omega \tilde{\Omega}^{-1}$:
\begin{align}\label{Condition:T}
    \varepsilon \omega \tilde{\Omega}^{-1} = \mathbf{D} (\mathbf{u}^\varepsilon) \left[ \mathbf{m}_2'(\mathbf{u}^\varepsilon) - \left( \mathbf{f}'(\mathbf{u}^\varepsilon) \right)^2 \right]^{-1}.
\end{align}
Since $\mathbf{D}$ is in general not invertible\footnote{This is for example the case of the Navier-Stokes equation, where the first line of $\mathbf{D}$, related to mass conservation, is identically null.}, this relationship cannot be inverted to compute $\tilde{\Omega}$. However, as will be shown thereafter, this problem can be solved thanks to the use of specific temporal schemes. In particular, the formal limit $\mathbf{D} = \mathbf{0}$ can be considered by the present framework, allowing us to recover the particular non-viscous case of \cite{Torlo}.

From \eqref{eq:CE_transport_diffusion}, we see that the diffusive system \eqref{eq:transport_diffusion} can be approximated by the kinetic model under two assumptions:
\begin{enumerate}
	\item[(i)] the matrix $\mathbf{m}_2'(\mathbf{u}^\varepsilon) - \left( \mathbf{f}'(\mathbf{u}^\varepsilon) \right)^2$ is invertible (necessary to compute $ \varepsilon \omega \tilde{\Omega}^{-1} $ through \eqref{Condition:T}),
	\item[(ii)] the consistency error in \eqref{eq:CE_transport_diffusion} can be neglected, i.e. $\varepsilon \ll 1$.
\end{enumerate}
The first assumption will be justified in Sec.~\ref{sec:assumption_justification} for any wave model that satisfies the sub-characteristic condition. Regarding the second assumption, ensuring its validity is the key to the method we propose. It is therefore paramount to have a correct estimation of $\varepsilon$. Recalling that $\Vert \Omega \Vert = \mathcal{O}(1)$ and $\omega = \ell / \Vert \Lambda \Vert$, we have 
\begin{align}
	\varepsilon = \frac{\Vert \Lambda\Vert}{\ell}\left \Vert \mathbf{D} 
\left[ \mathbf{m}_2'(\mathbf{u}^\varepsilon) - \left( \mathbf{f}'(\mathbf{u}^\varepsilon) \right)^2 \right]^{-1} \right \Vert.
\end{align}
Noticing that $\mathbf{m}_2'(\bbu^\varepsilon)$ is proportional to $\Vert \Lambda \Vert^2$, we can observe that
\begin{align}
	\varepsilon \approx \frac{\Vert \mathbf{D} \Vert}{\Vert \Lambda \Vert \ell},
	\label{eq:def_epsilon}
\end{align}
which is the general definition of Knudsen number we will adopt for all the examples of sections \ref{sec:application_scalars} and \ref{sec:application_NS}.
Note that, looking at how we have obtained \eqref{Condition:T}, there is in fact no need that $\varepsilon$ be a scalar, it can be a diagonal matrix, i.e. we can have a Knudsen number for each line of $\mathbf{D}$. The dependence of $\varepsilon$ on $\Vert\Lambda\Vert$ provides an interesting feature to the consistency error: it can be arbitrarily adjusted by modifying the kinetic velocities, which are a free parameter as far as the monotonicity condition of the Maxwellian is satisfied (in general, the sub-characteristic condition). This property will be exhibited in the numerical validations of Secs.~\ref{sec:application_scalars}-\ref{sec:application_NS}. The dependence of $\varepsilon$ on $\ell$ recalls us that the validity of the hypothesis always depends on the characteristic scale of the problem under consideration. It is very similar to the validity of the Navier-Stokes equations, which can be reasonably adopted as far as the characteristic length of a problem is larger than the mean free path of particles (continuum assumption).

Before discussing this on a case by case basis, let us check that for any wave model that satisfies the sub-characteristic condition, our assumption (i) is justified.
\subsection{Justification for the construction of the collision matrix}
\label{sec:assumption_justification}
The calculations have been performed under the assumption (i) that the matrix $\left[ \mathbf{m}_2'(\mathbf{u}^\varepsilon) - \left( \mathbf{f}'(\mathbf{u}^\varepsilon) \right)^2 \right]$ is invertible. In the present section, we provide a rationale for it. In the particular case of the two-wave model, assumption it is satisfied as a consequence of the sub-characteristic condition, as shown in the example below.

\begin{example}[Two-wave model]
    For a two-wave model with velocities $(-a,a)$, we have $\Lambda^2 = a^2 \mathbf{I}_{kp}$, so $\bbm_2'(\bbu^\varepsilon) =  a^2 \mathbf{I}_p$. Then, noting $\mathbf{f}'(\bbu^\varepsilon) = \mathbf{Q} \mathbf{R} \mathbf{Q}^{-1}$ where $\mathbf{R}=\mathrm{diag}(\lambda_1,..,\lambda_p)$ is a diagonal matrix, we have
    \begin{align}
        \mathbf{m}_2'(\mathbf{u}^\varepsilon) - \left( \mathbf{f}'(\mathbf{u}^\varepsilon) \right)^2 = a^2 \mathbf{I}_p + \mathbf{Q} \mathbf{R}^2 \mathbf{Q}^{-1} = \mathbf{Q}\, \mathrm{diag} (a^2-\lambda_1^2, \dots, a^2 - \lambda_p^2) \,\mathbf{Q}^{-1}.
    \end{align}
    When the subcharacteristic condition $\mathrm{max}_i |\lambda_i | < a$ is satisfied, the matrix $\mathbf{m}_2'(\mathbf{u}^\varepsilon) - \left( \mathbf{f}'(\mathbf{u}^\varepsilon) \right)^2$ is diagonalizable with strictly positive eigenvalues and invertible. 
\end{example}

We now prove  that this property can be generalized to any wave system, assuming that $\mathbf{D}$ satisfies the properties given in Sec.~\ref{sec:problem_statement} and that the sub-characteristic condition is satisfied. We first have the following proposition.

\begin{proposition}\label{prop:1}
    Suppose that there exists a strictly convex entropy $\eta (\bbu)$ with Hessian matrix $\mathbf{A}_0$ such that:
    \begin{enumerate}
        \item  $\mathbf{A}_0\bbf'(\bbu) $ is symmetric, 
        \item The Maxwellians are monotone: $\mathbf{A}_0\mathbb{M}_i'(\bbu) $ is symmetric positive definite for all $i \in [1, k]$ as in \cite{Bouchut}.
        \item The entries $a_i$ of the diagonal matrix $\Lambda$ satisfy $\min_i\vert a_i\vert > \rho\big (\bbf'(\bbu)\big )$.
    \end{enumerate}
    Then the matrix $\mathbf{K} =  \mathbf{A}_0\left[ \mathbf{m}_2'(\bbu) - (\bbf'(\bbu))^2 \right]$ is symmetric positive semi-definite.
\end{proposition}
\begin{remark}
In practice, condition 2, i.e. the monotonicity of the Maxwellians, implies condition 3.
\end{remark}

\begin{proof}
We first show that $\mathbf{K}$ is symmetric. We have:
\begin{align}
    \mathbf{K} = \mathbf{A}_0\P \Lambda^2 \M'(\bbu)  - \mathbf{A}_0\bbf'(\bbu)^2 .
    \label{eq:demo1_eq1}
\end{align}
With $\Lambda = \mathrm{diag}(a_1 \mathbf{I}_p, \dots a_k \mathbf{I}_p)$, we have $\mathbf{A}_0\P \Lambda^2 \M'(\bbu)  = \sum_{i=1}^k a_i^2 \mathbf{A}_0 \M_i'(\bbu) $. Hence, since  $\mathbf{A}_0\M_i'(\bbu) $ is symmetric, the first term of (\ref{eq:demo1_eq1}) is symmetric. Regarding the second term, using the fact that $\mathbf{A}_0$ and $\mathbf{A}_0\bbf'(\bbu) $ are symmetric, we have:
\begin{equation*}
\begin{split}
    \mathbf{A}_0\bbf'(\bbu)^2  &=  \big ( \mathbf{A}_0\bbf'(\bbu) \big ) \bbf'(\bbu)
    = \big ( \mathbf{A}_0\bbf'(\bbu) \big )^T\bbf'(\bbu)
    = \bbf'(\bbu)^T \mathbf{A}_0 \bbf'(\bbu)\\
    & = \bbf'(\bbu)^T\big ( \mathbf{A}_0 \bbf'(\bbu)\big )^T
   =\big ( \bbf'(\bbu)^2\big )^T \mathbf{A}_0\\
    &=\big ( \mathbf{A}_0\bbf'(\bbu)^2 \big )^T
\end{split}
\end{equation*}
so that the second term of (\ref{eq:demo1_eq1}) is symmetric. Hence $\mathbf{K}$ is symmetric.
Next, we denote by $\langle \mathbf{x},\mathbf{y}\rangle$ the Euclidian scalar product between the vectors $\mathbf{x}$ and $\mathbf{y}$.
Using $\bbu=\sum\limits_i \M_i(\bbu)$ so that $\mathbf{I}_p=\sum\limits_i \M'_i(\bbu)$, we have for any $\mathbf{x}$
\begin{equation*}
\begin{split}
\langle \mathbf{K}\mathbf{x},\mathbf{x}\rangle&=\sum_{i=1}^k a_i^2 \langle \mathbf{A}_0 \M_i'(\bbu)
\mathbf{x},\mathbf{x}\rangle-\langle \mathbf{A}_0\bbf'(\bbu)^2\mathbf{x},\mathbf{x}\rangle\\
&\geq \min\limits_i\vert a_i\vert^2\langle \sum_i \mathbf{A}_0\M_i'(\bbu)\mathbf{x},\mathbf{x}\rangle -
\langle \mathbf{A}_0\bbf'(\bbu)^2\mathbf{x},\mathbf{x}\rangle\\
& > \langle \mathbf{A}_0\bigg (  \rho(\bbf'(\bbu))^2 \mathbf{I}_p-\bbf'(\bbu)^2\bigg )\mathbf{x},\mathbf{x}\rangle.
\end{split}
\end{equation*}
This shows that, with $\mathbf{x}=\mathbf{A}_0^{-1/2} \mathbf{y}$,
  \begin{equation*}
\begin{split}
\langle \mathbf{K}\mathbf{A}_0^{-1/2}\mathbf{y},\mathbf{A}_0^{-1/2}\mathbf{y}\rangle& > \langle \mathbf{A}_0\bigg (  \rho(\bbf'(\bbu))^2 \mathbf{I}_p-\bbf'(\bbu)^2\bigg )\mathbf{A}_0^{-1/2}\mathbf{y},\mathbf{A}_0^{-1/2}\mathbf{y}\rangle\\
&=\langle \mathbf{A}_0^{-1/2}\bigg [\mathbf{A}_0\bigg (  \rho(\bbf'(\bbu))^2 \mathbf{I}_p-\bbf'(\bbu)^2\bigg )\mathbf{A}_0^{-1/2}\bigg ]\mathbf{y},\mathbf{y}\rangle.
\end{split}
\end{equation*}

We notice that the symmetric matrix
$$\mathbf{A}_0^{-1/2}\bigg [\mathbf{A}_0\bigg (  \rho(\bbf'(\bbu))^2 \mathbf{I}_p-\bbf'(\bbu)^2\bigg )\bigg ]\mathbf{A}_0^{-1/2}=
  \rho(\bbf'(\bbu))^2 \mathbf{I}_p- \big ( \mathbf{A}_0^{1/2}\bbf'(\bbu) \mathbf{A}_0^{-1/2}\big )^2$$
  has positive eigenvalues because the eigenvalues of $\mathbf{A}_0^{1/2}\bbf'(\bbu) \mathbf{A}_0^{-1/2}$ are those of $\bbf'(\bbu)$. Hence,
$$ \langle \mathbf{K}\mathbf{A}_0^{-1/2}\mathbf{y},\mathbf{A}_0^{-1/2}\mathbf{y}\rangle >0.$$
We take $\mathbf{y}=\mathbf{A}_0^{1/2}\mathbf{x}$ and we obtain the result.
\end{proof}
\begin{corollary}
    If all the conditions of Proposition \ref{prop:1} are satisfied, then the matrix $\mathbf{m}_2'(\bbu) - (\bbf'(\bbu))^2$ has real strictly positive eigenvalues and is invertible.
\end{corollary}

\begin{proof} We have 
$$\mathbf{A}_0^{-1/2}\mathbf{K}\mathbf{A}_0^{-1/2}=\mathbf{A}_0^{1/2}\bigg ( \mathbf{m}_2'-\big ( \bbf'(\bbu)\big )^2 \bigg ) \mathbf{A}_0^{-1/2}
$$
so that  $\mathbf{m}_2'-\big ( \bbf'(\bbu)\big )^2$ is similar to $\mathbf{A}_0^{-1/2}\mathbf{K}\mathbf{A}_0^{-1/2}$ and that has positive eigenvalues.
\end{proof}
\begin{corollary}
    Le us assume that $\mathbf{A}_0$ is such that   $\mathbf{D} \mathbf{A}_0^{-1}$ is symmetric and has positive eigenvalues.  Then the matrix $\varepsilon \tilde{\Omega}^{-1}$ given by (\ref{Condition:T}) has real non-negative eigenvalues.
\end{corollary}
Note that assuming the symmetry of $\mathbf{A}_0 \mathbf{D}$ or $ \mathbf{D} \mathbf{A}_0^{-1}$ is equivalent since $\mathbf{A}_0\big ( \mathbf{D} \mathbf{A}_0^{-1}\big ) \mathbf{A}_0=\mathbf{A}_0 \mathbf{D}$.
\begin{proof}Recall that $\mathbf{K} =\mathbf{A}_0 \left[ \mathbf{m}_2'(\bbu) - (\bbf'(\bbu))^2 \right] $
is symmetric positive definite,  so that $$ \mathbf{A}_0^{-1}\mathbf{K}\mathbf{A}_0^{-1}=\left[ \mathbf{m}_2'(\bbu) - (\bbf'(\bbu))^2 \right]\mathbf{A}_0^{-1}:=\mathbf{P}$$ is symmetric with positive eigenvalues.
 We have from Eq.~(\ref{eq:condition_with_MRT})
\begin{equation*}
 \mathbf{P}^{-1/2} \left( \varepsilon \omega \tilde{\Omega}^{-1} \right) \mathbf{P}^{1/2} = \mathbf{P}^{-1/2} \left(\mathbf{D} \mathbf{A}_0^{-1} \right)\mathbf{P}^{-1/2}
\end{equation*}
so that $ \varepsilon \omega \tilde{\Omega}^{-1}$ is similar to $\mathbf{P}^{-1/2} \mathbf{D} \mathbf{A}_0 \mathbf{P}^{-1/2}$ which is symmetric positive semi-definite. Then $\varepsilon \tilde{\Omega}^{-1}$ is diagonalizable with non-negative eigenvalues.
\end{proof}

\section{Time and space discretization: arbitrary high-order method}

In this work, we adopt the numerical discretization developed in~\cite{Torlo}. We present it for the collision matrix model (\ref{eq:kinetic_MRT}), knowing that the more common system (\ref{eq:kinetic_equation_retained}) can be recovered in the particular case $\Omega = \mathbf{I}_{kp}$. It relies on two ingredients. The first one is a defect correction (DeC) strategy that allows us to construct schemes with a given accuracy independent of the relaxation matrix. The second ingredient is the spatial discretization which is similar to what is done in \cite{Torlo} and inspired by \cite{iserle}.  The outcome is a scheme that is of order $q$ in space and time, independently of the relaxation parameter. The integer $q$ can be chosen arbitrarily.

\subsection{Time discretization: deferred correction IMEX method}

We want to have a robust and accurate time integration of (\ref{eq:kinetic_MRT}). In order to get rid off the stiffness induced by the relaxation term, the idea, already described in \cite{Torlo},  is to introduce two operators for solving \eqref{eq:kinetic_equation_adim}, $\mathcal{L}_1$ and $\mathcal{L}_2$ such that
\begin{enumerate}
\item the $\mathcal{L}_i=\delta_t^i \bF+\delta_x^i \bF-\mathbf{S}_i$ write as a sum of a temporal contribution, $\delta_t^i \bF$ that approximates the time derivative, a spatial contribution $\delta_x^i \bF$ that approximates $\Lambda \dpar{\bF}{x}$ and a source term $\mathbf{S}^i$ for the relaxation term,
\item $\mathbf{S}^1=\mathbf{S}^2=\mathbf{S}$, $\P \mathbf{S}=0$,
\item $\mathcal{ L}_2(\bF)=0$ solves \eqref{eq:kinetic_MRT} with $q$-th order,
\item $\P\mathcal{L}_1(\bF)=0$ is explicit in time,
\item  for any $\bF$, $\mathcal{L}_2(\bF)-\mathcal{L}_1(\bF)=O(\Delta t)$. 
\end{enumerate}
In practice, we will see that the property $\mathbf{S}^1=\mathbf{S}^2=\mathbf{S}$, $\P \mathbf{S}=0$ is key in establishing this approximation property because there is no stiff term in $\mathcal{L}_2(\bF)-\mathcal{L}_1(\bF)$. Defining $t_n$ the discrete time at time step $n$, the operators $\mathcal{L}_i$ depend on $\bF(t_n)$ and possibly on previous time steps. 
Then, as shown in \cite{Torlo}, the algorithm
\begin{enumerate}
\item $\bF^{(0)}=\bF(t_n)$,
\item $\bF^{(p+1)}$ solution of $\mathcal{L}_1(\bF^{(p+1)})=\mathcal{L}_1(\bF^{(p)})-\mathcal{L}_2(\bF^{(p)})$,
\end{enumerate}
is such that $\bF^{(q)}-\bF(t_{n+1})=O(\Delta t^q)$ where $\bF(t_{n+1})$ is the solution of $\mathcal{L}_2(\bF)=0$ at time $t_{n+1}=t_n+\Delta t$. 

 In the present section, we first introduce a first-order IMEX scheme that can be made fully explicit. Then, following \cite{Torlo}, we introduce general explicit high-order schemes based on implicit Runge-Kutta integrations together with a deferred correction algorithm. In all this section, we drop the specification of the space variable $x$, knowing that every operations are local in space except for the discrete derivation $\delta_x^i \bF$ which will be discussed in Sec.~\ref{sec:space_disc}.

\subsubsection{First-order IMEX scheme}
\label{sec:first_order_IMEX}

Following~\cite{Torlo}, we use a first-order explicit integration for the convective part, and a first-order implicit integration for the collision term which behaves as a stiff term. Integrating between $t_n$ and $t_{n+1} = t_n + \Delta t$, this reads
\begin{align}
    \bF (t_{n+1}) - \bF(t_n) + \Delta t \Lambda \delta_x^i \bF (t_n) = \Delta t \, \frac{\Omega }{\omega \varepsilon}(\P \bF (t_{n+1})) \left[ \M(\P \bF (t_{n+1}) - \bF(t_{n+1}) \right],
\end{align}
where we recall that, using \eqref{Condition:T}, the matrix $\Omega/(\omega \varepsilon)$ depends on the solution $\P \bF$  which is here evaluated at time $t_{n+1}$. This scheme is implicit, but can be made fully explicit by first applying the projector $\P$ to the solution at time $t_{n+1}$, leading to
\begin{align}
    \P \bF(t_{n+1}) = \P \bF (t_n) - \Delta t \P \Lambda \delta_x^i \bF(t_n),
\end{align}
so that $\M(\P \bF(t_{n+1}))$ and $\Omega/(\omega \varepsilon)(\P \bF(t_{n+1}))$ can be explicitly computed. Then defining $\hat{\Omega}_{n+1} = \Delta t \, \Omega/(\omega \varepsilon) (\P \bF (t_{n+1}))$, $\bF(t_{n+1})$ can be explicitly computed by reversing a linear system leading to:
\begin{align}
    \bF(t_{n+1}) = \left[ \mathbf{I}_{kp} + \hat{\Omega}_{n+1}^{-1} \right]^{-1} \left\{ \hat{\Omega}_{n+1}^{-1} \left[ \bF(t_n) -\Delta t \Lambda \delta_x^i \bF(t_n) \right] + \M (\P \bF(t_{n+1})) \right\}.
\end{align}
Note that this scheme only involves $\hat{\Omega}_{n+1}$ by its inverse matrix $\hat{\Omega}_{n+1}^{-1}$, which can be computed even when $\mathbf{D}$ is not inversible by \eqref{Condition:T}.

In this section, we have described a first-order method in time and space that is explicit. It does not need to use the DeC method. For higher order in time method, we do need DeC, so we need an operator $\mathcal{L}_1$. The spatial and temporal approximation will be the same as here, however the relaxation term will be approximated by  the same approximation as for the $\mathcal{L}_2$ operator, to be defined in the following section.

\subsubsection{High-order: IMEX Runge-Kutta schemes with deferred correction}\label{sec:RK}

We now want to build robust arbitrary high-order schemes for \eqref{eq:kinetic_MRT}. To this extent, let us rewrite the semi-discrete system as 
\begin{align}
    \frac{\mathrm{d} \bF}{\mathrm{d} t} = - \Lambda \delta_x^i \bF + \frac{\Omega}{\omega \varepsilon} (\P \bF) \left( \M(\P \bF) - \bF \right) \equiv \mathcal{F}(\bF).
\end{align}
This system of ODE can be numerically discretized using implicit RK methods of order $q$ in time, considering $s$ sub-time nodes denoted as $c_1 = 0 < c_2 < .. < c_s = 1$. Knowing the solution as time $t_n$, the updated one at time $t_{n+1} = t_n + \Delta t$ is given by:
\begin{align}
    \forall j \in \{1, \dots s \}, \qquad & \bF_j = \bF (t_n) + \Delta t \sum_{k=1}^s a_{jk} \mathcal{F} (\bF_k), \label{eq:implicit_RK} \\
    & \bF(t_{n+1}) = \bF(t_n) + \Delta t \sum_{k=1}^s b_k \mathcal{F}(\bF_k),
\end{align}
where $a_{jk}$ and $b_k$ are appropriate coefficients depending on the scheme under consideration. Coefficients $a_{jk}$ are related to the subtime nodes $c_j$ through the following consistency condition:
\begin{align}
    \forall j \in \{1, \dots s \},\qquad \sum_{k=1}^s a_{jk} = c_j.
\end{align}
Also note that the last step of this generalized RK scheme, involving $b_k$, is fully explicit. Therefore, we will only focus on Eq.~(\ref{eq:implicit_RK}). This brings us to define the following vectors of size $kps$:
\begin{align}
    \hat{\bF} = 
    \begin{pmatrix}
        \bF_1 \\ \vdots \\ \bF_s
    \end{pmatrix}, \qquad 
    \hat{\bF}_0 =
    \begin{pmatrix}
        \bF(t_n) \\ \vdots \\ \bF(t_n)
    \end{pmatrix}, \qquad
    \hat{\M} = 
    \begin{pmatrix}
        \M (\P \bF_1) \\ \vdots \\ \M (\P \bF_s)
    \end{pmatrix},
\end{align}
together with the following matrices in $\mathcal{M}_{kps}(\R)$:
\begin{align}
    \hat{\mathbf{A}} =  \Delta t \, 
    \mathbf{A} \otimes \mathbf{I}_{kp}, \qquad
    \hat{\Omega} =     \begin{pmatrix}
        \Omega/(\omega \varepsilon) (\P \bF_1) & \hdots & 0 \\
        \vdots & \ddots & \vdots \\
        0 & \hdots & \Omega/(\omega \varepsilon) (\P \bF_s)
    \end{pmatrix}, \qquad
    \hat{\Lambda} = 
    \mathbf{I}_s \otimes \Lambda.
\end{align}
With these notations, Eq.~(\ref{eq:implicit_RK}) reads:
\begin{align}
    \hat{\bF}  = \hat{\bF}_0 - \hat{\mathbf{A}} \hat{\Lambda} \delta_x^i \hat{\bF} + \hat{\mathbf{A}} \hat{\Omega} (\hat{\M} - \hat{\bF}),
\end{align}
which leads us to define a high-order operator $\mathcal{L}^2$ acting on $\hat{\bF}$ as
\begin{align}
    \mathcal{L}^2(\hat{\bF}) = \hat{\bF}  - \hat{\bF}_0 + \hat{\mathbf{A}} \hat{\Lambda} \delta_x^i \hat{\bF} - \hat{\mathbf{A}} \hat{\Omega} (\hat{\M} - \hat{\bF}).
\end{align}
The high-order scheme simply reads $\mathcal{L}^2 (\hat{\bF}) = 0$. However, this scheme is not explicit, a priori because of two terms: (1) the transport term $\hat{\Lambda} \delta_x^i \hat{\bF}$ and (2) the collision term $\hat{\Omega} (\hat{\M} - \hat{\bF})$. In fact, as mentioned in \cite{Torlo} and in the same was as with the first-order IMEX scheme, the implicitness of the collision term vanishes after applying the projector $\P$ to $\mathcal{L}^2$. However, the implicitness of the transport term remains. To address it, we use a deferred correction scheme, consisting in the iterative resolution of an explicit problem involving a low-order scheme $\mathcal{L}^1$. In the present context, we define $\mathcal{L}^1$ as:
\begin{align}
    \mathcal{L}^1(\hat{\bF}) = \hat{\bF} - \hat{\bF}_0 + \hat{\mathbf{C}} \hat{\Lambda} \delta_x^i \hat{\bF}_0 - \hat{\mathbf{A}} \hat{\Omega} (\hat{\M} - \hat{\bF}),
\end{align}
where
\begin{align}
    \hat{\mathbf{C}} = 
    \begin{pmatrix}
        c_1 \mathbf{I}_{kp} & \hdots & 0 \\
        \vdots & \ddots & \vdots \\
        0 & \hdots & c_s \mathbf{I}_{kp}
    \end{pmatrix}.
\end{align}
Based on it, the principle of the deferred correction algorithm reads:
\begin{enumerate}
	\item We define:
\begin{align}
    \hat{\bF}^{(0)} \equiv \hat{\bF}_0.
\end{align}
	\item The following iterative scheme is solved:
\begin{align}
	\forall p \in \{ 0, \dots, M-1 \}, \qquad \mathcal{L}^1 (\hat{\bF}^{(p+1)}) = \mathcal{L}^1 (\hat{\bF}^{(p)}) - \mathcal{L}^2 (\hat{\bF}^{(p)}). \label{eq:iteration_DeC}
\end{align}
	\item The updated solution at time $t+\Delta t$ is obtained by setting
\begin{align}
	\bF(t_{n+1}) = \bF(t_n) + \Delta t \sum_{k=1}^s b_k \mathcal{F} (\bF_k^{(p+1)}).
\end{align}
\end{enumerate}
Note that in many implicit RK schemes (\textit{e.g.} Lobato IIIA, Lobato IIIC, see section \ref{sec:lobatto} bellow), we have $\forall k \in \{ 1, \dots, s \},\ b_k = a_{sk}$, so that the last step of the algorithm can be reduced to
\begin{align}
    \bF(t_{n+1}) = \bF^{(p+1)}_s.
    \label{eq:RK_simplified_update}
\end{align}
It can be shown that this iterative scheme has a formal error of $\Delta t^{\min(q, M)}$. Hence, by taking $M=q$, the order of convergence of the implicit RK scheme is recovered.

Using the definitions of $\mathcal{L}^1$ and $\mathcal{L}^2$, we have
\begin{align}
    \mathcal{L}^1 (\hat{\bF}^{(p)}) - \mathcal{L}^2 (\hat{\bF}^{(p)}) = \hat{\mathbf{C}} \hat{\Lambda} \delta_x^i \hat{\bF}_0 - \hat{\mathbf{A}} \hat{\Lambda} \delta_x \hat{\bF}^{(p)},
\end{align}
so that (\ref{eq:iteration_DeC}) yields
\begin{align}
     \hat{\bF}^{(p+1)} - \hat{\mathbf{A}} \hat{\Omega}^{(p+1)} (\hat{\M}^{(p+1)} - \hat{\bF}^{(p+1)}) = \hat{\bF}_0  - \hat{\mathbf{A}} \hat{\Lambda} \delta_x \hat{\bF}^{(p)}.
     \label{eq:DeC_scheme}
\end{align}
This scheme is implicit, but can be made explicit by first applying the projector:
\begin{align}
    \P \hat{\bF}^{(p+1)} = \P \hat{\bF}_0  - \P \hat{\mathbf{A}} \hat{\Lambda} \delta_x \hat{\bF}^{(p)},
    \label{eq:DeC_macros}
\end{align}
such that $\hat{\M}^{(p+1)}$ and $\hat{\Omega}^{(p+1)}$ can be explicitly computed, and then reversing the following linear system:
\begin{align}
    \left[ \mathbf{I}_{kps} + \hat{\mathbf{A}} \hat{\Omega}^{(p+1)} \right] \hat{\bF}^{(p+1)} = \hat{\bF}_0 - \hat{\mathbf{A}} \hat{\Lambda} \delta \hat{\bF}^{(p)} + \hat{\mathbf{A}} \hat{\Omega}^{(p+1)}\hat{\M}^{(p+1)}.
    \label{eq:DeC_implicit_to_reverse}
\end{align}
Dropping the exponent $(p+1)$ on $\hat{\Omega}$ for the sake of convenience, the solution can be written as:
\begin{align}
    \hat{\bF}^{(p+1)} = \hat{\Omega}^{-1} \left[ \hat{\Omega}^{-1} + \hat{\mathbf{A}} \right]^{-1} \left( \hat{\bF}_0 - \hat{\mathbf{A}}\hat{\Lambda} \delta \hat{\bF}^{(p)} \right) + \hat{\mathbf{A}} \left[ \hat{\Omega}^{-1} + \hat{\mathbf{A}} \right]^{-1} \hat{\M}^{(p+1)}.
    \label{eq:DeC_explicit}
\end{align}
As for the proposed first-order IMEX scheme, this scheme only involves $\hat{\Omega}$ through its inverse matrix $\hat{\Omega}^{-1}$, which can be computed even when $\mathbf{D}$ is not invertible \textit{via} \eqref{Condition:T}. However, a condition for solving this problem is that the matrix $\hat{\Omega}^{-1} + \hat{\mathbf{A}}$ must be invertible, which may not always be the case. For instance, when considering the Navier-Stokes equations for fluid dynamics, the absence of diffusion affecting mass conservation implies that the first row of $\hat{\Omega}^{-1}$ is null. Furthermore, if an implicit RK scheme like Lobato IIIA is used, the first row of $\hat{\mathbf{A}}$ is also null~\cite{Hairer}. In this simple case, the matrix $\hat{\Omega}^{-1} + \hat{\mathbf{A}}$ is not invertible. Hence, we will focus on schemes where the first row of $\mathbf{A}$ is non-null. This implies that, as seen in Eq.~(\ref{eq:implicit_RK}), even the first sub-time node $c_1=0$ is reconstructed, resulting in $\hat{\bF}_1 \neq \hat{\bF}(t_n)$. 
The Lobato IIIC scheme is an exemple of such RK methods, it will be the one adopted in the following~\cite{Hairer}. Note that using the DeC algorithm with a Lobato IIIC scheme can also be interpreted as an arbitrary derivative (ADER) method~\cite{Veiga2023}.

Interestingly, the non-diffusive case ($\mathbf{D}=\mathbf{0})$ can be recovered as the formal limit $\hat{\Omega}^{-1} = \mathbf{0}$, leading to the very simple update of populations:
\begin{align}
    \hat{\mathbf{F}}^{(p+1)} = \hat{\mathbb{M}}^{(p+1)}.
\end{align}

\subsubsection{Examples of  schemes}\label{sec:lobatto}

Below are some particular examples of Lobato IIIC schemes (from~\cite{Hairer}).

\paragraph{Second-order scheme} We consider the following Lobato IIIC second-order scheme ($q=2$) with  with two sub-time nodes ($s=2$) and
\begin{align}
    \mathbf{A} = 
    \begin{pmatrix}
        1/2 & -1/2 \\
        1/2 & 1/2
    \end{pmatrix}, 
    \qquad \mathbf{b} =
    \begin{pmatrix}
        1/2 & 1/2
    \end{pmatrix}.
    \label{eq:LobatoIIIC_O2}
\end{align}

\paragraph{Fourth-order scheme} We consider the following Lobato IIIC fourth-order scheme ($q=4$) with  with three sub-time nodes ($s=3$):
\begin{align}
    \mathbf{A} = 
    \begin{pmatrix}
        1/6 & -1/3 & 1/6 \\
        1/6 & 5/12 & -1/12 \\
        1/6 & 2/3 & 1/6
    \end{pmatrix}, \qquad
    \mathbf{b} = 
    \begin{pmatrix}
        1/6 & 2/3 & 1/6
    \end{pmatrix}.
    \label{eq:LobatoIIIC_O4}
\end{align}

 \paragraph{Sixth-order scheme}
 We consider the following Lobato IIIC sixth-order scheme ($q=6$) with four sub-time nodes ($s=4$):
 \begin{align}
     \mathbf{A} = 
     \begin{pmatrix}
         1/12 & -\sqrt{5}/12 & \sqrt{5}/12 & -1/12 \\
         1/12 & 1/4 & (10-7\sqrt{5})/60 & \sqrt{5}/60 \\
         1/12 & (10+7\sqrt{5})/60 & 1/4 & -\sqrt{5}/60 \\
         1/12 & 5/12 & 5/12 & 1/12
     \end{pmatrix}, \qquad 
     \mathbf{b} = 
     \begin{pmatrix}
         1/12 & 5/12 & 5/12 & 1/12
     \end{pmatrix}.
 \end{align}

Note that with these RK schemes, coefficients $b_k$ are equal to the last line of $\mathbf{A}$ so that the last step of the implicit RK scheme is redundant and Eq.~(\ref{eq:RK_simplified_update}) can be used. In the following, we will only focus on first-, second- and fourth-order integrations. The extension to higher-order methods is straightforward.

\subsection{Space discretization}
\label{sec:space_disc}

As discussed in~\cite{Torlo}, the only question left  to define a stable numerical scheme is to find numerical discretizations $\delta_x^i$ ensuring the stability of the convection (collisionless) scheme, assuming that the relaxation terms introduce diffusion. In the present work, we consider the space discretizations previously adopted in~\cite{abgrall2021preliminary} and inspired from \cite{iserle}, recalled below. We note $f_i$ a population being advected at a kinetic velocity $a_i$ of $\Lambda$ and $\Delta x$ is the uniform mesh size.

\paragraph{First-order ($\delta_x^1$)} We use the upwind scheme:
\begin{align}
    \delta_x^1 f_i (x,t) = 
    \begin{cases}
        [ f_i(x, t) - f_i(x-\Delta x, t) ]/\Delta x \qquad \mathrm{if}\ a_i \geq 0, \\
        [ f_i(x+\Delta x, t) - f_i(x,t) ]/\Delta x \qquad \mathrm{else}.
    \end{cases}
\end{align}

\paragraph{Second-order ($\delta_x^2$)} We define:
\begin{align}
    \delta_x^2 f_i(x,t) =
    \begin{cases}
        [ f_i(x+\Delta x, t)/3 + f_i(x, t)/2 - f_i(x-\Delta x, t) + f_i(x-2\Delta x, t)/6 ] /\Delta x \qquad & \mathrm{if}\ a_i \geq 0, \\
        [-f_i(x-\Delta x, t)/3 - f_i(x, t)/2 + f_i(x+\Delta x, t) - f_i(x+2\Delta x, t)/6 ] /\Delta x & \mathrm{else}.
    \end{cases}
\end{align}

\paragraph{Fourth-order ($\delta_x^4$)} We define:
\begin{align}
    \delta_x^4 f_i(x,t) = \frac{1}{12 \Delta x} \left[ f_i(x-2\Delta x, t) - f_i(x+2\Delta x, t) \right] + \frac{2}{3 \Delta x} \left[ f_i(x+\Delta x, t) - f_i(x-\Delta x, t) \right].
\end{align}

Regarding the fourth-order discretization, since the space derivative operator is independent of the considered wave, note that the numerical method can be equivalently recast as a scheme acting on moments of the populations $\mathbf{F}$, i.e. on variables $(\bbu^\varepsilon, \bbv^\varepsilon)$. This observation may be considered for improving the efficiency of the fourth-order scheme. 

The stability of the ensuing numerical schemes based on Lobato IIIC time discretizations is investigated in the following section.

\subsection{Linear stability analysis}

In this section, the linear stability of the transport term of the kinetic model is investigated. We therefore focus on the following simplified 1D transport equation,
\begin{align}
    \frac{\partial y}{\partial t} = - a \frac{\partial y}{\partial x},
\end{align}
where $y: \mathbb{R}^+ \times \mathbb{R} \rightarrow \mathbb{R}$ is a differentiable function of time and space and $a>0$ is an advection velocity. Eventually performing a Fourier transform in space, we define $\hat{y}(k,t) = \int y(x,t) e^{-ikx} \mathrm{d} x$ where $k \in \mathbb{R}$ is a wavenumber. After discretizing time in sub-steps $\{ t_n,\ n \in \mathbb{N} \}$ and space in points $\{ x_j,\ j \in \mathbb{Z} \}$ with uniform time step $\Delta t$ and mesh size $\Delta x$, we note $y_j^n$ the solution of the numerical scheme at $(t_n, x_j)$ and $\hat{y}^n$ its Fourier transform. Considering the discretized space derivative $\delta_x^i$, the Fourier transform of $\delta_x^i y^n_j$ is $g \hat{y}^n/\Delta x$ with:
\begin{align}
    & \mathrm{First-order}\ (\delta_x^1): & g = 1-e^{- \mathrm{i} \theta}, \\
    & \mathrm{Second-order}\ (\delta_x^2): & g = \frac{1}{3} e^{\mathrm{i} \theta} + \frac{1}{2} - e^{-\mathrm{i} \theta} + \frac{1}{6} e^{2 \mathrm{i} \theta},\\
    & \mathrm{Fourth-order}\ (\delta_x^4): & g = \mathrm{i} \left( \frac{4}{3} \sin(\theta) -\frac{1}{6} \sin(2 \theta) \right),
\end{align}
where $\theta = k \Delta x \in \mathbb{R}$. An amplification factor can be defined as $G=\hat{y}^{n+1}/\hat{y}^n$ and absolute stability is ensured provided that $|G| \leq 1$ for any $k \in \mathbb{R}$. Following these notations, numerical stability of the implicit $\mathcal{L}^2$ operator and of the DeC algorithm are investigated below for first-, second- and fourth-order time integrations.

\subsubsection{First-order time integration}

The first-order IMEX scheme proposed in Sec.~\ref{sec:first_order_IMEX} is based on an explicit forward Euler time integration for the transport term. This reads
\begin{align}
    \frac{\hat{y}^{n+1}-\hat{y}^n}{\Delta t} = -a \frac{g}{\Delta x} \hat{y}^n,
\end{align}
so that the amplification factor is
\begin{align}
    G = 1+z, \qquad z = -\lambda g,
\end{align}
where $\lambda = a \Delta t/\Delta x$ is the CFL number. The stability criterion of the explicit Euler time integration is $|1+z| \leq 1$ and the relation $z=-\lambda g$ eventually provides restrictions on the CFL number $\lambda$ to satisfy this criterion, depending on the space discretization characterized by $g$. The stability region in the complex plane together with the possible values of $z$ for different CFL numbers and space discretizations are displayed in Fig.~\ref{fig:stability_O1}. With the operator $\delta_x^1$, a necessary and sufficient condition for the stability of this scheme is $\lambda \leq 1$. With $\delta_x^2$ and $\delta_x^4$, this scheme is unconditionally unstable since stability can only be ensured for $\lambda=0$. For $\delta_x^4$, the instability can simply be observed by the fact that $z \in \mathrm{i} \mathbb{R}$, so that the stability condition $|1+z| \leq 1$ can only be met for $z=0$.

\begin{figure}[h!]
    \begin{subfigure}{0.33\textwidth}
    \includegraphics{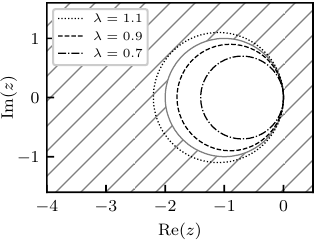}
    \caption{$\delta_x^1$}
    \end{subfigure}
    \begin{subfigure}{0.33\textwidth}
    \includegraphics{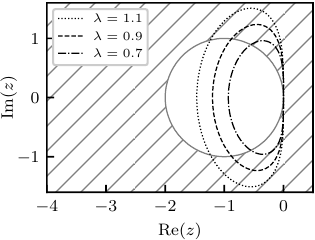}
    \caption{$\delta_x^2$}
    \end{subfigure}
    \begin{subfigure}{0.33\textwidth}
    \includegraphics{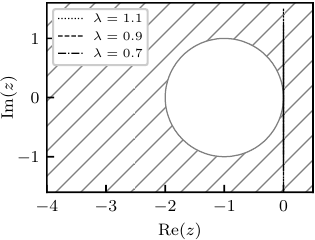}
    \caption{$\delta_x^4$}
    \end{subfigure}
    \caption{Stability plots of the explicit Euler time integration for transport equation with space discretizations $\delta_x^1$, $\delta_x^2$ and $\delta_x^4$. Hashed area: instability zone of the time integration scheme ($|G|>1$). Dashed lines: possible values of $z$ for varying CFL numbers $\lambda$ and space discretization operators.}
    \label{fig:stability_O1}
\end{figure}

\subsubsection{Second-order time integration}

We now focus on the $\mathcal{L}^2$ time integration given by the second-order Lobato IIIC scheme of Eq.~(\ref{eq:LobatoIIIC_O2}). In the Fourier space, the scheme reads
\begin{align}
    \hat{\mathbf{y}}^{n+1} = \hat{\mathbf{y}}_0^n - \lambda g \mathbf{A} \hat{\mathbf{y}}^{n+1},
\end{align}
where $\hat{\mathbf{y}}^{n+1}$ is a vector of size $s=2$ whose components are the Fourier transforms of the solution at each updated sub-time node (the last line is equal to $\hat{y}^{n+1}$) and $\hat{\mathbf{y}}_0^{n} = [\hat{y}^n, \hat{y}^n]^T$. Inverting the implicit system yields
\begin{align}
    \hat{\mathbf{y}}^{n+1} = \left[ \mathbf{Id} - z \mathbf{A} \right]^{-1} \hat{\mathbf{y}}_0^n,
\end{align}
where $z=-\lambda g$ and
\begin{align}
    \left[ \mathbf{Id} - z \mathbf{A} \right]^{-1} = \frac{1}{z^2-2z+2} 
    \begin{bmatrix}
        2-z & -z \\
        z & 2-z
    \end{bmatrix}.
\end{align}
The amplification factor is obtained by summing up the components of the last row of this matrix, which yields
\begin{align}
    G = \frac{2}{z^2-2z+2}, \qquad z = -\lambda g.
\end{align}
Stability curves obtained for this scheme are displayed in Fig.~\ref{fig:stability_O2_L2} for different $\delta$ operators. The A-stability of the Lobato IIIC scheme is recovered, leading to an unconditional stability in terms of CFL number.

\begin{figure}[h!]
    \begin{subfigure}{0.33\textwidth}
    \includegraphics{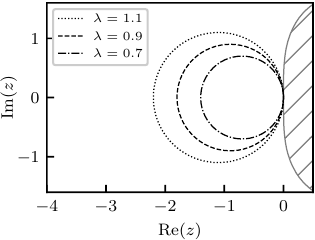}
    \caption{$\delta_x^1$}
    \end{subfigure}
    \begin{subfigure}{0.33\textwidth}
    \includegraphics{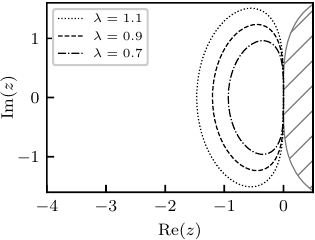}
    \caption{$\delta_x^2$}
    \end{subfigure}
    \begin{subfigure}{0.33\textwidth}
    \includegraphics{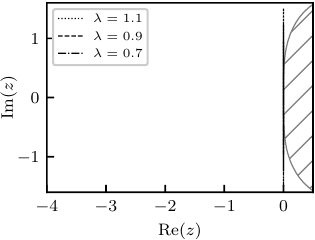}
    \caption{$\delta_x^4$}
    \end{subfigure}
    \caption{Stability plots of the $\mathcal{L}^2$ operator with second-order Lobato IIIC time integration for transport equation with space discretizations $\delta_x^1$, $\delta_x^2$ and $\delta_x^4$. Hashed area: instability zone of the time integration scheme ($|G|>1$). Dashed lines: possible values of $z$ for varying CFL numbers $\lambda$ and space discretization operators.}
    \label{fig:stability_O2_L2}
\end{figure}

Let us now consider the DeC algorithm applied to this scheme. The iterations read:
\begin{align}
    & \hat{\mathbf{y}}^{n+1, (0)} = \hat{\mathbf{y}}^{n}_0, \nonumber \\
    & \forall p \in \llbracket 0, M-1 \rrbracket, \qquad \hat{\mathbf{y}}^{n+1, (p+1)} = \hat{\mathbf{y}}_0^n - \lambda g \mathbf{A} \hat{\mathbf{y}}^{n+1, (p)}, \nonumber \\
    & \hat{\mathbf{y}}^{n+1} = \hat{\mathbf{y}}^{n+1, (M)}.
\end{align}
For two iterations ($M=2$), the scheme can be written in the following compact form:
\begin{align}
    \hat{\mathbf{y}}^{n+1} = \left[ \mathbf{Id} + z \mathbf{A} + z^2 \mathbf{A}^2 \right] \hat{\mathbf{y}}_0^n,
\end{align}
where
\begin{align}
    \mathbf{Id} + z \mathbf{A} + z^2 \mathbf{A}^2 = 
    \begin{bmatrix}
        1 + z/2 & -(z+z^2)/2 \\
        (z+z^2)/2 & 1+z/2
    \end{bmatrix}.
\end{align}
The amplification factor is obtained by summing up the components of the last line of this matrix, which yields
\begin{align}
    G = 1 + z + \frac{z^2}{2}, \qquad z=-\lambda g.
\end{align}
Stability curves are displayed for this scheme in Fig.~\ref{fig:stability_O2_DeC}. With the $\delta_x^1$ operator, a necessary and sufficient condition for stability is $\lambda \leq 1$. With $\delta_x^2$, a slightly lower CFL number can be reached ($\lambda < 0.87$). With $\delta_x^4$, this scheme is unconditionally unstable.

Maximal CFL numbers obtained for this scheme and for different numbers of iterations of the DeC algorithm are compiled in Table~\ref{tab:max_CFL_LobatoIIIC}.

\begin{figure}[h!]
    \begin{subfigure}{0.33\textwidth}
    \includegraphics{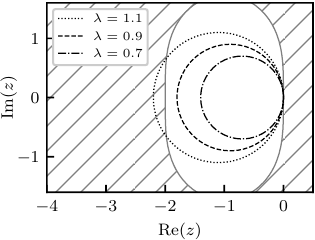}
    \caption{$\delta_x^1$}
    \end{subfigure}
    \begin{subfigure}{0.33\textwidth}
    \includegraphics{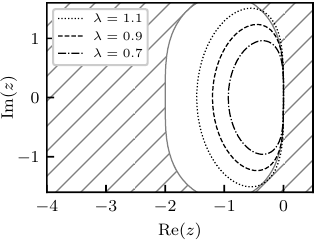}
    \caption{$\delta_x^2$}
    \end{subfigure}
    \begin{subfigure}{0.33\textwidth}
    \includegraphics{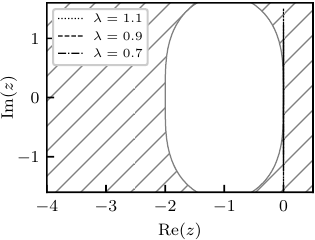}
    \caption{$\delta_x^4$}
    \end{subfigure}
    \caption{Stability plots of DeC time integration based on second-order Lobato IIIC for transport equation with space discretizations $\delta_x^1$, $\delta_x^2$ and $\delta_x^4$. Hashed area: instability zone of the time integration scheme ($|G|>1$). Dashed lines: possible values of $z$ for varying CFL numbers $\lambda$ and space discretization operators.}
    \label{fig:stability_O2_DeC}
\end{figure}

\subsubsection{Fourth-order time integration}

We now focus on the $\mathcal{L}_2$ algorithm involving the fourth-order Lobato IIIC scheme of Eq.~(\ref{eq:LobatoIIIC_O4}). Compared to its second-order counterpart, the only modification is the matrix $\mathbf{A}$ which leads to
\begin{align}
    \left[ \mathbf{Id} - z \mathbf{A} \right]^{-1} = \frac{1}{z^3-6z^2+18z-24}
    \begin{bmatrix}
        -3z^2 + 14z - 24 & -4z^2 + 8z & z^2-4z \\
        z^2 - 4z & 8z-24 & -z^2 + 2z \\
        -z^2 - 4z & 4z^2 - 16z & -3z^2 + 14z -24
    \end{bmatrix}.
\end{align}
The amplification factor is given by
\begin{align}
    G = \frac{-6z-24}{z^3-6z^2+18z-24}, \qquad z = -\lambda g.
\end{align}
Stability curves obtained for this scheme are displayed in Fig.~\ref{fig:stability_O4_L2} for different $\delta_x^i$ operators. As for its second-order counterpart, the A-stability of the Lobato IIIC scheme is recovered, leading to an unconditional stability in terms of CFL number.

\begin{figure}[h!]
    \begin{subfigure}{0.33\textwidth}
    \includegraphics{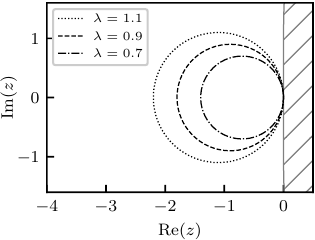}
    \caption{$\delta_x^1$}
    \end{subfigure}
    \begin{subfigure}{0.33\textwidth}
    \includegraphics{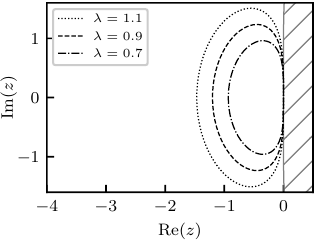}
    \caption{$\delta_x^2$}
    \end{subfigure}
    \begin{subfigure}{0.33\textwidth}
    \includegraphics{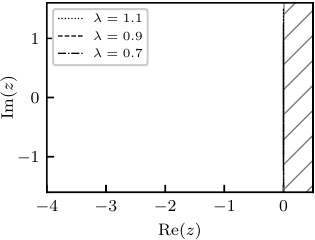}
    \caption{$\delta_x^4$}
    \end{subfigure}
    \caption{Stability plots of the $\mathcal{L}^2$ operator with fourth-order Lobato IIIC time integration for transport equation with space discretizations $\delta_x^1$, $\delta_x^2$ and $\delta_x^4$. Hashed area: instability zone of the time integration scheme ($|G|>1$). Dashed lines: possible values of $z$ for varying CFL numbers $\lambda$ and space discretization operators.}
    \label{fig:stability_O4_L2}
\end{figure}

The DeC scheme with four iterations ($M=4$) reads:
\begin{align}
    \hat{\mathbf{y}}^{n+1} = \left[ \mathbf{Id} + z \mathbf{A} + z^2 \mathbf{A}^2 + z^3 \mathbf{A}^3 + z^4 \mathbf{A}^4 \right] \hat{\mathbf{y}}_0^n,
\end{align}
where
\begin{align}
     & \mathbf{Id} + z \mathbf{A} + z^2 \mathbf{A}^2 + z^3 \mathbf{A}^3 + z^4 \mathbf{A}^4 = \frac{1}{576} \times \nonumber \\
     & \
     \begin{bmatrix}
         4z^4 + 96z+576 & 13z^4+12z^3-48z^2-192z & z^4+12z^3+48z^2+96z \\
         z^4 + 12z^3 + 48z^2 + 96z & (-23z^4 - 36z^3 + 144z^2 + 960z +2304)/4 & (13z^4 + 12z^3 - 48z^2 - 192z)/4 \\
         16z^4 + 48z^3 + 96z^2 + 96z & 4z^4 + 48z^3 + 192z^2 + 384z & 4z^4 + 96z + 576
     \end{bmatrix}.
\end{align}
The amplification factor is given by
\begin{align}
    G = 1+z+\frac{z^2}{2} + \frac{z^3}{6} + \frac{z^4}{24}, \qquad z=-\lambda g.
\end{align}
We recover a result recently demonstrated in~\cite{micalizzi2023new}: the amplification function of the ADER scheme is $G=\sum_{k=0}^M z^k/k!$. Stability curves obtained for this scheme are displayed in Fig.~\ref{fig:stability_O4_DeC} for different $\delta_x^i$ operator. We see that in any case, $\lambda >1$ can be reached. Detailed results of maximal CFL numbers are summarized in Table~\ref{tab:max_CFL_LobatoIIIC} depending on the number of iterations of the DeC algorithm.

\begin{figure}[h!]
    \begin{subfigure}{0.33\textwidth}
    \includegraphics{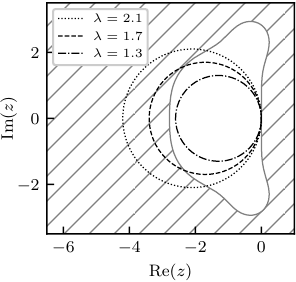}
    \caption{$\delta_x^1$}
    \end{subfigure}
    \begin{subfigure}{0.33\textwidth}
    \includegraphics{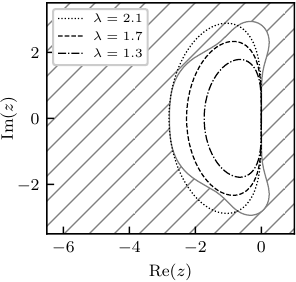}
    \caption{$\delta_x^2$}
    \end{subfigure}
    \begin{subfigure}{0.33\textwidth}
    \includegraphics{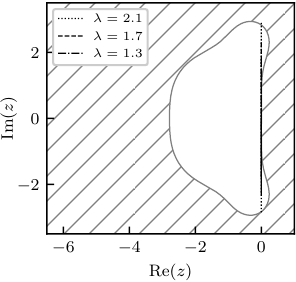}
    \caption{$\delta_x^4$}
    \end{subfigure}
    \caption{Stability plots of the DeC time integration based on fourth-order Lobato IIIC for transport equation with space discretizations $\delta_x^1$, $\delta_x^2$ and $\delta_x^4$. Hashed area: instability zone of the time integration scheme ($|G|>1$). Dashed lines: possible values of $z$ for varying CFL numbers $\lambda$ and space discretization operators.}
    \label{fig:stability_O4_DeC}
\end{figure}

\begin{table}[!ht]
  \centering
\begin{tabular}{|p{0.1\textwidth}|>{\centering}p{0.05\textwidth}||>{\centering}p{0.05\textwidth}|>{\centering}p{0.05\textwidth}|>{\centering}p{0.05\textwidth}|>{\centering}p{0.05\textwidth}|>{\centering}p{0.05\textwidth}|>{\centering\arraybackslash}p{0.05\textwidth}|}
\hline
\multicolumn{2}{|c|}{Scheme} & \multicolumn{6}{c|}{\# iterations} \\
\hline
Order & $\delta$ & 1 & 2 & 3 & 4 & 5 & 6 \\
\hline
2 & $\delta_x^1$ & 1 & 1 & 1 & 0.78 & 0.71 & 0.85 \\
2 & $\delta_x^2$ & 0 & 0.87 & 0.87 & 0.96 & 0.88 & 0.98 \\
2 & $\delta_x^4$ & 0 & 0 & 0 & 0.66 & 1.03 & 1.16 \\
\hline
4 & $\delta_x^1$ & 1 & 1 & 1.26 & 1.39 & 1.46 & 1.34\\
4 & $\delta_x^2$ & 0 & 0.87 & 1.63 & 1.75 & 1.81 & 1.77\\
4 & $\delta_x^4$ & 0 & 0 & 1.26 & 2.06 & 0.04 & 0.62\\
\hline
\end{tabular}
\caption{Critical CFL numbers $\lambda$ of Lobato IIIC schemes.}
\label{tab:max_CFL_LobatoIIIC}
\end{table}

Furthermore, for the sake of completeness and comparisons, similar stability analyses are performed with Lobato IIIA schemes of second and fourth orders~\cite{Hairer}. Maximal CFL numbers are compiled in Table~\ref{tab:max_CFL_LobatoIIIA}. Even though the stability can be affected by the choice of RK scheme, we see that when the minimal number iterations is performed, similar stability criteria are obtained with Lobato IIIA and Lobato IIIC. A result demonstrated in~\cite{micalizzi2023new} is recovered here: the DeC algorithm involving $M$ iterations of a $M^{th}$-order implicit RK scheme leads to the same stability function, whatever the implicit RK scheme. We conclude that the use of Lobato IIIC instead of Lobato IIIA does not affect the numerical stability.

\begin{table}[!ht]
  \centering
\begin{tabular}{|p{0.1\textwidth}|>{\centering}p{0.05\textwidth}||>{\centering}p{0.05\textwidth}|>{\centering}p{0.05\textwidth}|>{\centering}p{0.05\textwidth}|>{\centering}p{0.05\textwidth}|>{\centering}p{0.05\textwidth}|>{\centering\arraybackslash}p{0.05\textwidth}|}
\hline
\multicolumn{2}{|c|}{Scheme} & \multicolumn{6}{c|}{\# iterations} \\
\hline
Order & $\delta$ & 1 & 2 & 3 & 4 & 5 & 6 \\
\hline
2 & $\delta_1$ & 1 & 1 & 1 & 1 & 1 & 1\\
2 & $\delta_2$ & 0 & 0.87 & 1.22 & 1.02 & 1.08 & 1.24 \\
2 & $\delta_4$ & 0 & 0 & 1.46 & 1.46 & 0.03 & 0.07 \\
\hline
4 & $\delta_1$ & 1 & 1 & 1.26 & 1.39 & 1.77 & 1.77 \\
4 & $\delta_2$ & 0 & 0.87 & 1.63 & 1.75 & 2.06 & 2.06 \\
4 & $\delta_4$ & 0 & 0 & 1.26 & 2.06 & 2.52 & 2.52 \\
\hline
\end{tabular}
\caption{Critical CFL numbers $\lambda$ of Lobato IIIA schemes.}
\label{tab:max_CFL_LobatoIIIA}
\end{table}

\section{Application to scalar problems}
\label{sec:application_scalars}

We first assess the proposed method for the resolution of scalar problems in the form
\begin{align}
    \frac{\partial u}{\partial t} + \frac{\partial f(u)}{\partial x} = \alpha \frac{\partial^2 u}{\partial x},
\end{align}
where $u= u(x, t) \in \mathbb{R}$, $f:\mathbb{R} \rightarrow \mathbb{R}$ a convective flux and $\alpha \geq 0$ a constant diffusion parameter. In the present section, different expressions will be considered for the convective flux in order to solve (1) the diffusion equation, (2) the advection-diffusion equation, (3) the viscous Burgers equation. We first discuss on the adopted choice of waves for the kinetic model, then detail each equation under consideration. The purpose of this section is also to quantify the $\mathcal{O}(\varepsilon^2)$ consistency error inherent of the kinetic model, in order to propose a method for appropriately selecting the kinetic velocities in $\Lambda$.

In any case and following the stability analysis, the following CFL number are systematically considered:
\begin{itemize}
    \item First-order scheme (implicit Euler with $\delta_x^1$): $\lambda$ = 1,
    \item Second-order scheme (DeC with second-order Lobato IIIC, $\delta_x^2$): $\lambda = 0.8$,
    \item Fourth-order scheme (DeC with fourth-order Lobato IIIC, $\delta_x^4$): $\lambda = 2$.
\end{itemize}
Note that these CFL numbers are based on the advection velocity of the kinetic model $a$ ($\lambda = a \Delta t/\Delta x$) and are in general different from the standard definition of CFL number based on $|f'(u)|$ . To make it clear, the CFL number based on $a$ will be referred to as $\lambda$ and the one based on $|f'(u)|$ will be simply referred to as CFL.

\subsection{Wave model}

We consider the two-wave model of Natalini~\cite{Natalini} which makes the kinetic system equivalent to Jin-Xin model~\cite{Jin}. Using the notations of Example~\ref{ex:scalar}, the Maxwellian reads 
\begin{align}
    \mathbb{M}_1(u^\varepsilon) = \frac{1}{2} \left( u^\varepsilon - \frac{f(u^\varepsilon)}{a} \right), \qquad \mathbb{M}_2(u^\varepsilon) = \frac{1}{2} \left( u^\varepsilon + \frac{f(u^\varepsilon)}{a} \right).
\end{align}
The sub-characteristic condition $a > |f'(u^\varepsilon)|$ is a sufficient condition to make this model compatible with entropy inequalities. In this scalar case, the collision matrix simply reads $\Omega = \tilde{\Omega} \mathbf{I}_2$ where $\tilde{\Omega}$ is a scalar, and \eqref{Condition:T} leads to
\begin{align}
	\varepsilon \omega \tilde{\Omega}^{-1} = \frac{\alpha}{a^2 - f'(u^\varepsilon)^2}.
    \label{eq:relation_tau_alpha_scalar}
\end{align}
Note that the relaxation parameter of Example~\ref{ex:tau_scalar} is recovered if we set $\tau = \varepsilon \omega \tilde{\Omega}^{-1}$. Following Eq.~\eqref{eq:def_epsilon}, we define the Knudsen number as
\begin{align}
	\varepsilon= \frac{\alpha}{a \ell},
	\label{eq:def_knudsen_scalar}
\end{align}
where $\ell$ is a characteristic length that depends on the problem under consideration. 

\subsection{Diffusion equation}
\label{sec:diffusion_equation}

We first consider the parabolic diffusion equation and set: $f(u) = 0$. This example is of particular interest because the sub-characteristic condition does not provide us any particular constraint on the wave velocity $a$ (except that $a>0$). The wave velocity can therefore be arbitrarily chosen, which allows us to better highlight the consistency error in $\mathcal{O}(\varepsilon^2)$.

A 1D domain of size $L=1$ is initialized with
\begin{align}
    u(x,0) = 1 + 0.01 \exp \left( - \frac{(x-0.5)^2}{\delta^2} \right),
\end{align}
where $\delta = 0.1$. The diffusion coefficient is set to $\alpha=0.01$. The characteristic length of this problem is the standard deviation of the Gaussian function. Therefore, we take $\ell = \delta$ in the definition of $\varepsilon$
\eqref{eq:def_knudsen_scalar}.

Figure~\ref{fig:Diffusion_plot} displays the Gaussian shape obtained after diffusion at time $t=0.1$ with 100 points by the first-, second- and fourth-order methods and two values of $a$, leading to two values of the Knudsen number. They are compared with the exact solution, 
\begin{align}
    u_{exact}(x,t) = 1+0.01 \sqrt{\frac{1}{1+4\alpha t/\delta^2}} \exp \left( -\frac{(x-0.5)^2}{\delta^2 + 4\alpha t} \right).
\end{align}

For $a=0.5$, the numerical solution is under-diffused compared to the exact one, whatever the order of accuracy of the method. This is due to the non-negligible second-order consistency error in Knudsen number ($\varepsilon = 0.2$) which prevents us to converge to the right solution. However, when decreasing the Knudsen number to $\varepsilon=0.05$, a qualitatively good agreement of the second- and fourth-order schemes with the exact solution is observed. The first-order scheme results this time in an over-diffusion which can be attributed to numerical dissipation.

\begin{figure}[h!]
    \centering
    \begin{subfigure}{0.47\textwidth}
    \includegraphics{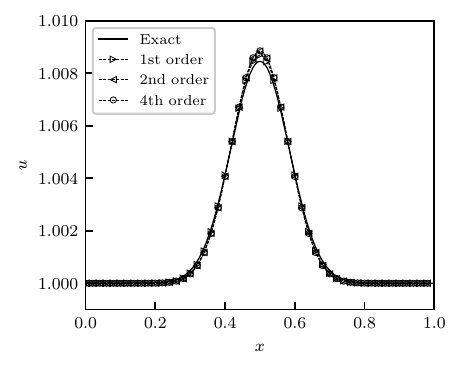}
    \caption{$a=0.5$ ($\varepsilon=0.2$)}
    \end{subfigure}
    \begin{subfigure}{0.47\textwidth}
    \includegraphics{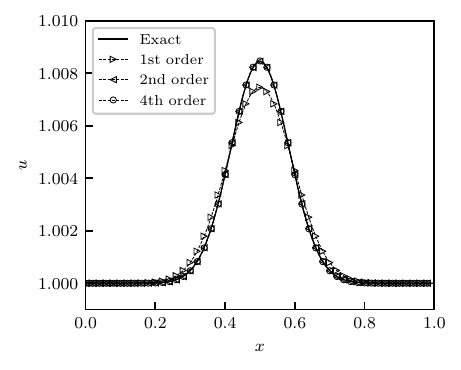}
    \caption{$a=2$ ($\varepsilon=0.05$)}
    \end{subfigure}
    \caption{Diffusion testcase of a Gaussian with $\alpha = 0.01$ at time $t=0.1$. Simulations are run with $100$ points for $x$ in $[0,1]$. Initial condition: $u_0(x) = 1+0.01\exp\left(-(x-0.5)^2/\delta^2 \right)$, $\delta = 0.1$. The effect of the change of wave velocity $a$ in the $\mathcal{O}(\varepsilon^2)$ consistency error is exhibited. Knudsen number is defined as $\varepsilon = \alpha/(a \delta)$.}
    \label{fig:Diffusion_plot}
\end{figure}

These observations can be quantified by performing a mesh convergence study for this test case at different values of $\varepsilon$ and measuring the $L^2$ error defined as
\begin{align}
    L^2 = \sqrt{\frac{\sum_i (u(x_i, T) - u_{exact}(x_i, T))^2}{\sum_i u_{exact}(x_i, T)^2}}.
\end{align}


\begin{figure}[h!]
    \begin{subfigure}{0.33\textwidth}
    \centering
    \includegraphics{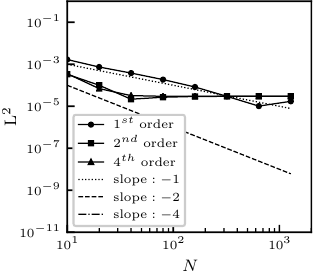}
    \caption{$a=1$ ($\varepsilon = 0.1$)}
    \end{subfigure}
    \begin{subfigure}{0.33\textwidth}
    \centering
    \includegraphics{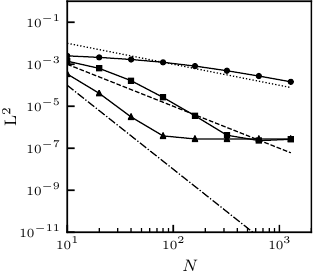}
    \caption{$a=10$ ($\varepsilon = 0.01$)}
    \end{subfigure}
    \begin{subfigure}{0.33\textwidth}
    \centering
    \includegraphics{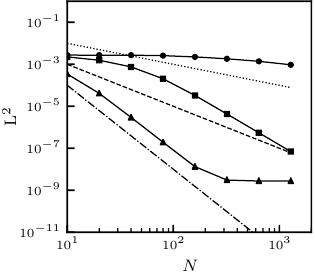}
    \caption{$a=100$ ($\varepsilon = 0.001$)}
    \end{subfigure}
    \caption{Mesh convergence study of the diffusion of an initial Gaussian shape with $\alpha=0.01$ at time $t=0.1$.}
    \label{fig:Order_Diffusion}
\end{figure}

\begin{table}[!ht]
\begin{subtable}{1.\textwidth}
\centering
\begin{tabular}{|p{0.05\textwidth}||>{\centering}p{0.16\textwidth}>{\centering}p{0.04\textwidth}|>{\centering}p{0.16\textwidth}>{\centering}p{0.04\textwidth}|>{\centering}p{0.16\textwidth}>{\centering\arraybackslash}p{0.04\textwidth}|}
\hline
 & \multicolumn{2}{|c|}{First-order} & \multicolumn{2}{c|}{Second-order} & \multicolumn{2}{c|}{Fourth-order} \\
\hline
$h$ & $L^2$ & $r$ & $L^2$ & $r$ & $L^2$ & $r$ \\
\hline
$40$ & $3.82224105\ 10^{-4}$ & - & $2.17227050\ 10^{-5}$ & - & $3.21519778\ 10^{-5}$ & - \\ 
$80$ & $1.86929660\ 10^{-4}$ & $1.03$ & $2.70144676\ 10^{-5}$ & $0.31$ & $2.99767769\ 10^{-5}$ & $0.10$ \\ 
$160$ & $8.28220611\ 10^{-5}$ & $1.17$ & $2.93761706\ 10^{-5}$ & $0.12$ & $2.98387193\ 10^{-5}$ & $0.01$ \\ 
$320$ & $2.96755607\ 10^{-5}$ & $1.48$ & $2.97652373\ 10^{-5}$ & $0.02$ & $2.98286604\ 10^{-5}$ & $0.00$ \\ 
$640$ & $1.02048119\ 10^{-5}$ & $1.54$ & $2.98200333\ 10^{-5}$ & $0.00$ & $2.98279336\ 10^{-5}$ & $0.00$ \\ 
$1280$ & $1.72434531\ 10^{-5}$ & $0.76$ & $2.98270094\ 10^{-5}$ & $0.00$ & $2.98278827\ 10^{-5}$ & $0.00$ \\ 
\hline
\end{tabular}
\caption{$a=1$ ($\varepsilon=0.1$)}
\end{subtable} \vspace{1mm} \\
\begin{subtable}{1.\textwidth}
\centering
\begin{tabular}{|p{0.05\textwidth}||>{\centering}p{0.16\textwidth}>{\centering}p{0.04\textwidth}|>{\centering}p{0.16\textwidth}>{\centering}p{0.04\textwidth}|>{\centering}p{0.16\textwidth}>{\centering\arraybackslash}p{0.04\textwidth}|}
\hline
 & \multicolumn{2}{|c|}{First-order} & \multicolumn{2}{c|}{Second-order} & \multicolumn{2}{c|}{Fourth-order} \\
\hline
$h$ & $L^2$ & $r$ & $L^2$ & $r$ & $L^2$ & $r$ \\
\hline
$40$ & $1.69529158\ 10^{-3}$ & - & $1.66059185\ 10^{-4}$ & - & $3.10403849\ 10^{-6}$ & - \\ 
$80$ & $1.24240266\ 10^{-3}$ & $0.45$ & $2.65664031\ 10^{-5}$ & $2.64$ & $3.89321748\ 10^{-7}$ & $3.00$ \\ 
$160$ & $8.22062271\ 10^{-4}$ & $0.60$ & $3.53718313\ 10^{-6}$ & $2.91$ & $2.78850475\ 10^{-7}$ & $0.48$ \\ 
$320$ & $4.94223899\ 10^{-4}$ & $0.73$ & $4.19746412\ 10^{-7}$ & $3.08$ & $2.74425526\ 10^{-7}$ & $0.02$ \\ 
$640$ & $2.75837103\ 10^{-4}$ & $0.84$ & $2.32517873\ 10^{-7}$ & $0.85$ & $2.74112551\ 10^{-7}$ & $0.00$ \\ 
$1280$ & $1.46530118\ 10^{-4}$ & $0.91$ & $2.66541049\ 10^{-7}$ & $0.20$ & $2.74087795\ 10^{-7}$ & $0.00$ \\ 
\hline
\end{tabular}
\caption{$a=10$ ($\varepsilon = 0.01$)}
\end{subtable} \\ \vspace{1mm} \\
\begin{subtable}{1.\textwidth}
\centering
\begin{tabular}{|p{0.05\textwidth}||>{\centering}p{0.16\textwidth}>{\centering}p{0.04\textwidth}|>{\centering}p{0.16\textwidth}>{\centering}p{0.04\textwidth}|>{\centering}p{0.16\textwidth}>{\centering\arraybackslash}p{0.04\textwidth}|}
\hline
 & \multicolumn{2}{|c|}{First-order} & \multicolumn{2}{c|}{Second-order} & \multicolumn{2}{c|}{Fourth-order} \\
\hline
$h$ & $L^2$ & $r$ & $L^2$ & $r$ & $L^2$ & $r$ \\
\hline
 $40$ & $2.71252475\ 10^{-3}$ & - & $7.52506198\ 10^{-4}$ & - & $2.95995370\ 10^{-6}$ & - \\ 
$80$ & $2.57969774\ 10^{-3}$ & $0.07$ & $2.04614875\ 10^{-4}$ & $1.88$ & $1.94227156\ 10^{-7}$ & $3.93$ \\ 
$160$ & $2.24326510\ 10^{-3}$ & $0.20$ & $3.28330005\ 10^{-5}$ & $2.64$ & $1.33685037\ 10^{-8}$ & $3.86$ \\ 
$320$ & $1.82794212\ 10^{-3}$ & $0.30$ & $4.30495069\ 10^{-6}$ & $2.93$ & $3.06947772\ 10^{-9}$ & $2.12$ \\ 
$640$ & $1.38363511\ 10^{-3}$ & $0.40$ & $5.46683143\ 10^{-7}$ & $2.98$ & $2.75468087\ 10^{-9}$ & $0.16$ \\ 
$1280$ & $9.48036540\ 10^{-4}$ & $0.55$ & $6.96842330\ 10^{-8}$ & $2.97$ & $2.73986802\ 10^{-9}$ & $0.01$ \\ 
\hline
\end{tabular}
\caption{$a=100$ ($\varepsilon = 0.001$)}
\end{subtable}
\caption{Orders of convergence for the diffusion problem and two-wave model for orders 1, 2
and 4. The final time is $t=0.1$ and the diffusion parameter is $\alpha=0.01$. The wave velocity $a$ is varied to exhibit its effect on the $\mathcal{O}(\varepsilon^2)$ consistency error, which appears as a plateau in the $L^2$ error of the high-order schemes.}
\label{tab:Order_Diffusion}
\end{table}

Convergence results of the $L^2$ errors obtained for meshes ranging from $N=10$ to $N=1280$ points and for three values of the Knudsen number are compiled in Table~\ref{tab:Order_Diffusion} and Fig~\ref{fig:Order_Diffusion}. The following observations can be drawn:
\begin{itemize}
    \item For a given Knudsen number, a plateau is systematically reached whatever the numerical method used, indicating a consistency error. The value of this plateau decreases as the Knudsen number decreases, which is in agreement with a $\mathcal{O}(\varepsilon^2)$ error.
    \item The numerical error of the first-order scheme increases as the Knudsen number decreases in agreement with the observations of Fig.~\ref{fig:Diffusion_plot}. Second- and fourth-order schemes do not seem to be affected by such a discrepancy.
    \item Interestingly, the second-order scheme seems to be hyper-convergent and exhibits a $(-3)$-slope when the Knudsen number is sufficiently small.
\end{itemize}

An asymptotic study of the consistency error is also performed on this test case. To this extent, simulations are done with the fourth-order scheme on a fine mesh with $1000$ points in order to get rid of numerical errors, and the Knudsen number is varied from $0.2$ to $0.00625$. $L^2$ errors and computed slopes $r$ are compiled in Table~\ref{tab:Order_Kn_Diffusion}. As expected, a $\mathcal{O}(\varepsilon^2)$ consistency error is exhibited.

\begin{table}[!ht]
\centering
\begin{tabular}{|p{0.03\textwidth}|>{\centering}p{0.13\textwidth}>{\centering}p{0.13\textwidth}>{\centering}p{0.13\textwidth}>{\centering}p{0.13\textwidth}>{\centering}p{0.13\textwidth}>{\centering\arraybackslash}p{0.13\textwidth}|}
\hline
$a$ & $0.5$ & $1$ & $2$ & $4$ & $8$ & $16$ \\
\hline
$\varepsilon$ & $0.2$ & $0.1$ & $0.05$ & $0.025$ & $0.0125$ & $0.00625$ \\
\hline 
$L^2$ & $1.397226\ 10^{-4}$ & $2.982789\ 10^{-5}$ & $6.982914\ 10^{-6}$ & $1.720013\ 10^{-6}$ & $4.284500\ 10^{-7}$ & $1.070190\ 10^{-7}$ \\
$r$ & - & $2.23$ & $2.09$ & $2.02$ & $2.01$ & $2.00$ \\
\hline
\end{tabular}
\caption{Asymptotic study of the consistency error in Knudsen number $\varepsilon$ of the diffusion of a Gaussian. Initial condition: $u_0(x) = 1+0.01\exp\left(-(x-0.5)^2/\delta^2 \right)$. Simulations are performed for $x$ in $[0,1]$ with $\delta=0.1$ and $\alpha=0.01$ up to time $t=0.1$. In order to get rid of numerical errors, a fine mesh of 1000 points is considered and simulations are performed with the fourth-order scheme. Knudsen number is defined as $\varepsilon = \alpha/(a \delta)$.}
\label{tab:Order_Kn_Diffusion}
\end{table}


\subsection{Advection-diffusion equation}

We now consider the advection-diffusion equation for which we set: $f(u)=cu$, where $c$ is a constant advection velocity. In the following, we reproduce the same test case as with the diffusion equation and set $c=10$ so that one cycle is made in the periodic domain at $t=0.1$. Note that the sub-characteristic conditions yields $a > 10$, so that, with $\alpha=0.01$, the Knudsen number is restricted to
\begin{align}
    \varepsilon < 0.01.
\end{align}
We see that in this case, the subcharacteristic condition is restrictive and allows us to a priori reasonably neglect the second-order error in $\varepsilon$. Fig.~\ref{fig:ADE_plot} displays the numerical solution obtained at $t=0.1$ with $N=100$ points for two values of $a$ satisfying the subcharacteristic condition: $a=12$ and $a=100$. The CFL numbers are given for each case in Table~\ref{tab:ADE_CFL}. With $a=12$, a good agreement of the second- and fourth-order methods is obtained with the exact solution, while the first-order one is more dissipative. With $a=100$, a similar observation as in Fig.~\ref{fig:Diffusion_plot} can be drawn: an increase of $a$ leads to an increase of the numerical error, especially for the first- and second-order method. With the fourth-order method, a good agreement is still observed with the exact solution.

\begin{table}[!ht]
\centering
\begin{tabular}{|p{0.05\textwidth}|>{\centering}p{0.1\textwidth}||>{\centering}p{0.15\textwidth}>{\centering}p{0.15\textwidth}>{\centering\arraybackslash}p{0.15\textwidth}|}
\hline
$a$ & $\varepsilon$ & CFL ($1^{st}$ order) & CFL ($2^{nd}$ order) & CFL ($4^{th}$ order) \\
\hline
$12$ & $0.0083$ & 0.83 & 0.67 & 1.67 \\
$100$ & $0.001$ & 0.1 & 0.08 & 0.2 \\
\hline
\end{tabular}
\caption{CFL numbers ($=c\Delta t/\Delta x$) for each case of Fig.~\ref{fig:ADE_plot}.}
\label{tab:ADE_CFL}
\end{table}

\begin{figure}[h!]
    \centering
    \begin{subfigure}{0.47\textwidth}
    \includegraphics{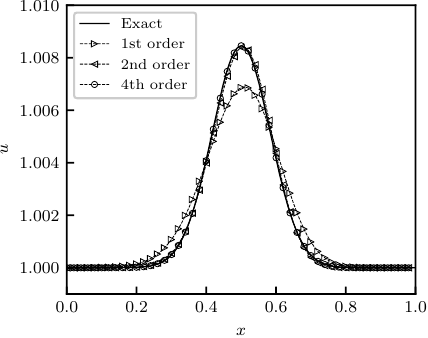}
    \caption{$a=12$ ($\varepsilon \approx 0.0083$)}
    \end{subfigure}
    \begin{subfigure}{0.47\textwidth}
    \includegraphics{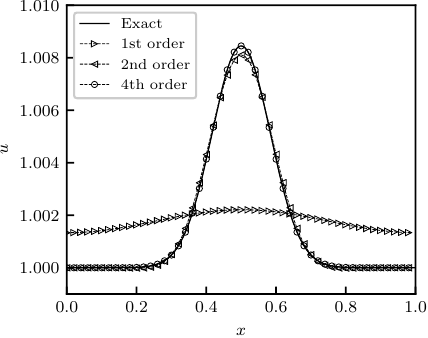}
    \caption{$a=100$ ($\varepsilon=0.001$)}
    \end{subfigure}
    \caption{Advection-diffusion testcase of a Gaussian with $c=10$ and $\alpha = 0.01$ at time $T=0.1$. Simulations are run with $100$ points for $x$ in $[0,1]$. Initial condition: $u_0(x) = 1+0.01\exp\left(-(x-0.5)^2/\delta^2 \right)$, $\delta = 0.1$.}
    \label{fig:ADE_plot}
\end{figure}

A mesh convergence study of this case is displayed in Fig.~\ref{fig:Order_ADE}, where the $L^2$ error is computed with the advected exact solution at time $t=0.005$. Similar observations as with the diffusion test case can be drawn: (1) a plateau is observed, whose value decreases when $\varepsilon$ decreases, (2) the numerical error of the first-order scheme increases when $\varepsilon$ decreases, (3) before reaching the plateau, the second-order scheme is hyperconvergent for $a=100$. Furthermore, the mesh convergence study performed in the inviscid case ($\alpha = 0$) with $a=12$ illustrates the asymptotic preservation of the method: no consistency error is observed in this case and the expected orders of convergence are correctly recovered. Quantitative results for this study are provided in Table~\ref{tab:Order_ADE}.

\begin{table}[!ht]
\begin{subtable}{1.\textwidth}
\centering
\begin{tabular}{|p{0.05\textwidth}||>{\centering}p{0.16\textwidth}>{\centering}p{0.04\textwidth}|>{\centering}p{0.16\textwidth}>{\centering}p{0.04\textwidth}|>{\centering}p{0.16\textwidth}>{\centering\arraybackslash}p{0.04\textwidth}|}
\hline
 & \multicolumn{2}{|c|}{First-order} & \multicolumn{2}{c|}{Second-order} & \multicolumn{2}{c|}{Fourth-order} \\
\hline
$h$ & $L^2$ & $r$ & $L^2$ & $r$ & $L^2$ & $r$ \\
\hline
$10$ & $7.06259159\ 10^{-4}$ & - & $5.28384218\ 10^{-4}$ & - & $5.81954092\ 10^{-4}$ & - \\ 
$20$ & $2.66013227\ 10^{-4}$ & $1.41$ & $1.30814558\ 10^{-4}$ & $2.01$ & $8.58949563\ 10^{-5}$ & $2.76$ \\ 
$40$ & $1.41383049\ 10^{-4}$ & $0.91$ & $3.53263068\ 10^{-5}$ & $1.89$ & $1.49109831\ 10^{-5}$ & $2.53$ \\ 
$80$ & $6.38941863\ 10^{-5}$ & $1.15$ & $1.18083985\ 10^{-5}$ & $1.58$ & $2.11115064\ 10^{-6}$ & $2.82$ \\ 
$160$ & $3.14424035\ 10^{-5}$ & $1.02$ & $5.30506366\ 10^{-6}$ & $1.15$ & $3.06479347\ 10^{-6}$ & $0.54$ \\ 
$320$ & $1.49633159\ 10^{-5}$ & $1.07$ & $3.65810496\ 10^{-6}$ & $0.54$ & $3.14047790\ 10^{-6}$ & $0.04$ \\ 
$640$ & $7.42427193\ 10^{-6}$ & $1.01$ & $3.26536195\ 10^{-6}$ & $0.16$ & $3.14562008\ 10^{-6}$ & $0.00$ \\ 
$1280$ & $4.11655580\ 10^{-6}$ & $0.85$ & $3.17390836\ 10^{-6}$ & $0.04$ & $3.14596946\ 10^{-6}$ & $0.00$ \\ 
\hline
\end{tabular}
\caption{$\alpha=0.01$, $a=12$ ($\varepsilon \approx 0.0083$)}
\end{subtable} \\ \vspace{1mm} \\
\begin{subtable}{1.\textwidth}
\centering
\begin{tabular}{|p{0.05\textwidth}||>{\centering}p{0.16\textwidth}>{\centering}p{0.04\textwidth}|>{\centering}p{0.16\textwidth}>{\centering}p{0.04\textwidth}|>{\centering}p{0.16\textwidth}>{\centering\arraybackslash}p{0.04\textwidth}|}
\hline
 & \multicolumn{2}{|c|}{First-order} & \multicolumn{2}{c|}{Second-order} & \multicolumn{2}{c|}{Fourth-order} \\
\hline
$h$ & $L^2$ & $r$ & $L^2$ & $r$ & $L^2$ & $r$ \\
\hline
$10$ & $2.49426713\ 10^{-3}$ & - & $1.38058136\ 10^{-3}$ & - & $5.81522375\ 10^{-4}$ & - \\ 
$20$ & $2.06086273\ 10^{-3}$ & $0.28$ & $6.56298958\ 10^{-4}$ & $1.07$ & $7.94325044\ 10^{-5}$ & $2.87$ \\ 
$40$ & $1.57008142\ 10^{-3}$ & $0.39$ & $1.68404811\ 10^{-4}$ & $1.96$ & $6.20792400\ 10^{-6}$ & $3.68$ \\ 
$80$ & $1.08437731\ 10^{-3}$ & $0.53$ & $2.65974144\ 10^{-5}$ & $2.66$ & $4.01522662\ 10^{-7}$ & $3.95$ \\ 
$160$ & $6.77927748\ 10^{-4}$ & $0.68$ & $3.48950919\ 10^{-6}$ & $2.93$ & $1.73049491\ 10^{-8}$ & $4.54$ \\ 
$320$ & $3.89725252\ 10^{-4}$ & $0.80$ & $4.46811263\ 10^{-7}$ & $2.97$ & $1.33013666\ 10^{-8}$ & $0.38$ \\ 
$640$ & $2.11113250\ 10^{-4}$ & $0.88$ & $7.16664468\ 10^{-8}$ & $2.64$ & $1.45034304\ 10^{-8}$ & $0.12$ \\ 
$1280$ & $1.10230164\ 10^{-4}$ & $0.94$ & $2.90811569\ 10^{-8}$ & $1.30$ & $1.45847006\ 10^{-8}$ & $0.01$ \\ 
\hline
\end{tabular}
\caption{$\alpha=0.01$, $a=100$ ($\varepsilon =0.001$)}
\end{subtable} \\ \vspace{1mm} \\
\begin{subtable}{1.\textwidth}
\centering
\begin{tabular}{|p{0.05\textwidth}||>{\centering}p{0.16\textwidth}>{\centering}p{0.04\textwidth}|>{\centering}p{0.16\textwidth}>{\centering}p{0.04\textwidth}|>{\centering}p{0.16\textwidth}>{\centering\arraybackslash}p{0.04\textwidth}|}
\hline
 & \multicolumn{2}{|c|}{First-order} & \multicolumn{2}{c|}{Second-order} & \multicolumn{2}{c|}{Fourth-order} \\
\hline
$h$ & $L^2$ & $r$ & $L^2$ & $r$ & $L^2$ & $r$ \\
\hline
$10$ & $7.29336300\ 10^{-4}$ & - & $5.43958682\ 10^{-4}$ & - & $5.84820313\ 10^{-4}$ & - \\ 
$20$ & $2.91964957\ 10^{-4}$ & $1.32$ & $1.39511213\ 10^{-4}$ & $1.96$ & $8.88746838\ 10^{-5}$ & $2.72$ \\ 
$40$ & $1.58939923\ 10^{-4}$ & $0.88$ & $3.03970848\ 10^{-5}$ & $2.20$ & $1.44536417\ 10^{-5}$ & $2.62$ \\ 
$80$ & $7.30983364\ 10^{-5}$ & $1.12$ & $7.75398598\ 10^{-6}$ & $1.97$ & $1.09585300\ 10^{-6}$ & $3.72$ \\ 
$160$ & $3.64904450\ 10^{-5}$ & $1.00$ & $1.96831591\ 10^{-6}$ & $1.98$ & $7.47569261\ 10^{-8}$ & $3.87$ \\ 
$320$ & $1.77341239\ 10^{-5}$ & $1.04$ & $4.94284325\ 10^{-7}$ & $1.99$ & $4.86169624\ 10^{-9}$ & $3.94$ \\ 
$640$ & $8.85906942\ 10^{-6}$ & $1.00$ & $1.23712855\ 10^{-7}$ & $2.00$ & $3.13704351\ 10^{-10}$ & $3.95$ \\ 
$1280$ & $4.39631371\ 10^{-6}$ & $1.01$ & $3.09370856\ 10^{-8}$ & $2.00$ & $1.96590137\ 10^{-11}$ & $4.00$ \\ 
\hline
\end{tabular}
\caption{$\alpha=0$, $a=12$ ($\varepsilon =0$)}
\end{subtable}
\caption{Orders of convergence for the advection-diffusion problem and two-wave model for orders 1, 2
and 4. The final time is $T=0.005$. The wave velocity $a$ is varied to exhibit its effect on the $\mathcal{O}(\varepsilon^2)$ consistency error, which appears as a plateau in the $L^2$ error of the high-order schemes.}
\label{tab:Order_ADE}
\end{table}

\begin{figure}[h!]
    \centering
    \begin{subfigure}{0.32\textwidth}
    \centering
    \includegraphics{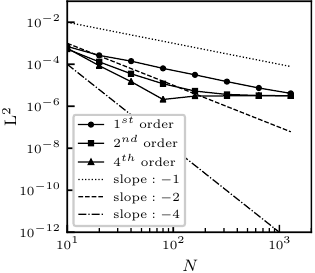}
    \caption{$\alpha=0.01$, $a=12$ ($\varepsilon \approx 0.0083$)}
    \end{subfigure}
    \begin{subfigure}{0.32\textwidth}
    \centering
    \includegraphics{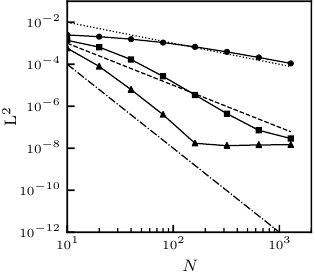}
    \caption{$\alpha=0.01$, $a=100$ ($\varepsilon = 0.001$)}
    \end{subfigure}
    \begin{subfigure}{0.32\textwidth}
    \centering
    \includegraphics{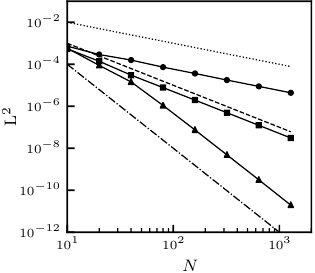}
    \caption{$\alpha=0$, $a=12$ ($\varepsilon = 0$)}
    \end{subfigure}
    \caption{Mesh convergence study of the advection-diffusion of an initial Gaussian shape at time $t=0.005$.}
    \label{fig:Order_ADE}
\end{figure}

Finally, an asymptotic study of the consistency error is also performed on this test case. Results are displayed in Table~\ref{tab:Order_Kn_ADE}. As for the diffusion equation, a clear $(-2)$-slope is observed in agreement with the expected $\mathcal{O}(\varepsilon^2)$ consistency error.

\begin{table}[!ht]
\centering
\begin{tabular}{|p{0.03\textwidth}|>{\centering}p{0.13\textwidth}>{\centering}p{0.13\textwidth}>{\centering}p{0.13\textwidth}>{\centering}p{0.13\textwidth}>{\centering}p{0.13\textwidth}>{\centering\arraybackslash}p{0.135\textwidth}|}
\hline
$a$ & $12$ & $24$ & $48$ & $96$ & $192$ & $384$ \\
\hline
$\varepsilon$ & $0.2$ & $0.1$ & $0.05$ & $0.025$ & $0.0125$ & $0.00625$ \\
\hline 
$L^2$ & $3.145929\ 10^{-6}$ & $3.024941\ 10^{-7}$ & $6.548333\ 10^{-8}$ & $1.582960\ 10^{-8}$ & $3.915344\ 10^{-9}$ & $9.667367\ 10^{-10}$ \\
$r$ & - & $3.38$ & $2.21$ & $2.05$ & $2.02$ & $2.02$ \\
\hline
\end{tabular}
\caption{Asymptotic study of the consistency error in Knudsen number $\varepsilon$ of the advection-diffusion of a Gaussian. Initial condition: $u_0(x) = 1+0.01\exp\left(-(x-0.5)^2/\delta^2 \right)$. Simulations are performed for $x$ in $[0,1]$ with $\delta=0.1$, $c=10$ and $\alpha=0.01$ up to time $t=0.005$. In order to get rid of numerical errors, a fine mesh of 1000 points is considered and simulations are performed with the fourth-order scheme.}
\label{tab:Order_Kn_ADE}
\end{table}

\subsection{Viscous Burgers equation}

We now want to solve the viscous Burgers equation, for which we set: $f(u) = u^2/2$. In this case, the sub-characteristic condition reads
\begin{align}
    a > \max_i|u(x_i)|.
\end{align}
Hence, contrary to the diffusion and advection-diffusion cases where a constant value of $a$ could be prescribed, it is here expected to vary over time. For this reason, the ratio $a/\max |u|$ will be prescribed in this section.

\subsubsection{Steady shock}

The first test case is a steady ``shock'' whose exact solution is given by~\cite{Benton1972}
\begin{align}
    u_{exact}(x) = - \frac{2\alpha}{\delta} \tanh ((x-L/2)/\delta),
\end{align}
where $\delta$ is the characteristic width of the shock. For this case, the Knudsen number is defined from \eqref{eq:def_knudsen_scalar} with $\ell = \delta$. We consider a domain of length $L=1$ discretized with $N=300$ points and set $\alpha=0.001$ and $\delta=0.01$. In order to evaluate the ability of the numerical method to converge towards the exact solution, we use a slightly modified initial condition:
\begin{align}
    u(x, 0) = -\frac{2 \alpha}{\delta} \tanh ((x-0.5) 10/\delta).
\end{align}
Dirichlet boundary conditions are used where distribution functions are simply set to the Maxwellian state corresponding to $u(x=0)=0.2$ on the left boundary and $u(x=1)=-0.2$ on the right boundary. Fig.~\ref{fig:Burgers_steadyshock} displays the numerical solutions obtained when time convergence is achieved for two ratios $a/\max|u|$. In the first case, the Knudsen number is $\varepsilon \approx 0.45$ so that the $\mathcal{O}(\varepsilon^2)$ cannot be neglected, which results in a mismatch with the exact solution. However, when $a$ increases, the Knudsen number can be artificially decreased so that a good agreement is observed with the exact solution for the second- and fourth-order schemes. Again, note that the numerical error of the first-order method considerably increases when $a$ increases.  

\begin{figure}[h!]
    \centering
    \begin{subfigure}{0.47\textwidth}
    \includegraphics{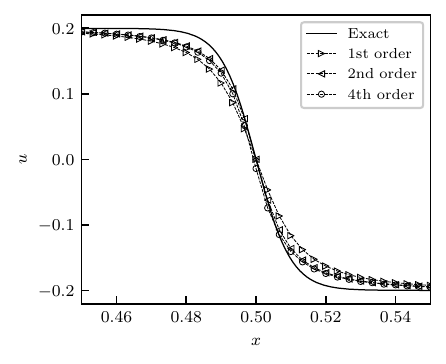}
    \caption{$a=1.1 \max(|u|)$ ($\varepsilon \approx 0.45$)}
    \end{subfigure}
    \begin{subfigure}{0.47\textwidth}
    \includegraphics{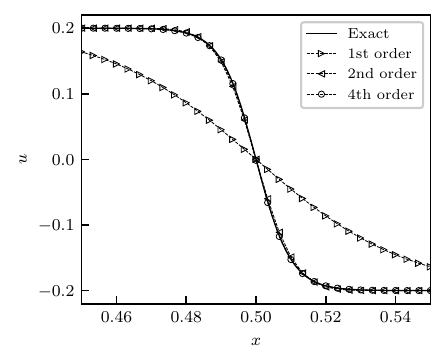}
    \caption{$a=10 \max(|u|)$ ($\varepsilon = 0.05$)}
    \end{subfigure}
    \caption{Steady ``shock'' testcase with the viscous Burgers equation with $\alpha = 0.001$. Simulations are run with $300$ points for $x$ in $[0,1]$. Exact solution: $u_{exact}(x)=-2 \alpha /\delta \tanh((x-0.5)/\delta)$, $\delta=1/100$. Initial condition: $u(x, 0) = -2 \alpha /\delta \tanh((x-0.5)\, 10/\delta)$. The Knudsen number is defined as $\varepsilon = \alpha/(a \delta)$.}
    \label{fig:Burgers_steadyshock}
\end{figure}

\subsubsection{Sinusoidal initialization}

We now consider a sinusoidal initialization of the $L=1$ domain as
\begin{align}
    u(x,0) = 0.5 + \sin(2\pi x).
\end{align}
The diffusion parameter is set to $\alpha = 0.01$ and $N=100$ points with periodic boundary conditions are considered for this case. This initialization is known to give birth to a viscous ``shock'' wave. An exact solution is given by~\cite{Benton1972} as
\begin{align}
    u_{exact}(x,t) = 0.5 + 2 \alpha \pi \ \frac{ \displaystyle 4 \sum_{n=1}^\infty n a_n e^{-4 \pi^2 \alpha n^2 t} \sin (2\pi n (x-0.5 t))}{\displaystyle  a_0 + 2 \sum_{n=1}^\infty a_n e^{-4\pi^2 \alpha n^2 t} \cos(2\pi n (x-0.5t))},
\end{align}
where
\begin{align}
    a_n = (-1)^n I_n \left( -\frac{1}{4 \pi \alpha} \right),
\end{align}
and where $I_n$ is the $n^{th}$-order exponentially scaled modified Bessel function of the first kind. In the following, we will consider the  first $100$ terms in the series, which provides us an accurate approximation of the exact solution. Numerical solutions obtained at time $t=0.5$ are displayed in Fig.~\ref{fig:Burgers_sin} and compared with the exact one. At this instant, a characteristic length of the viscous shock width can be built by measuring the distance between the maximal and the minimal values of the exact solution: $\delta \approx 0.12$. This characteristic length is used for the definition of the Knudsen number in \eqref{eq:def_knudsen_scalar}. Similar observations as for the steady viscous shock can be drawn.

\begin{figure}[h!]
    \centering
    \begin{subfigure}{0.47\textwidth}
    \includegraphics{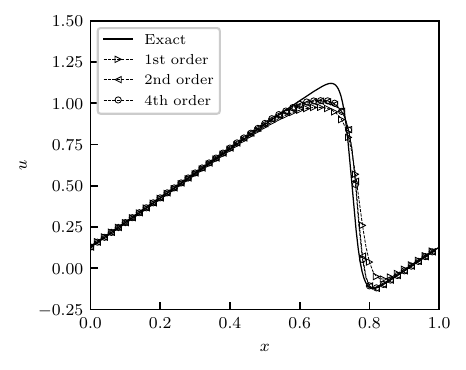}
    \caption{$a=1.1 \max(|u|)$ ($\varepsilon \approx 0.068$)}
    \end{subfigure}
    \begin{subfigure}{0.47\textwidth}
    \includegraphics{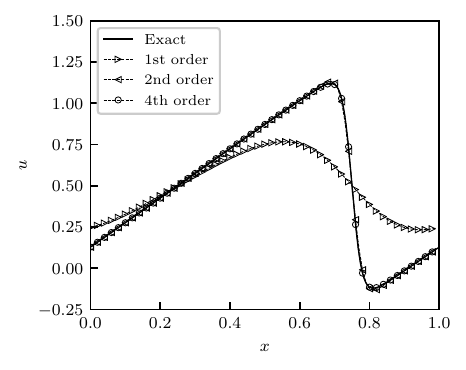}
    \caption{$a=10 \max(|u|)$ ($\varepsilon \approx 0.0074$)}
    \end{subfigure}
    \caption{Burgers equation with $\alpha = 0.01$ with the initial condition $u_0(x)=0.5+\sin(2\pi x)$ at $t=0.5$. Simulations are run with $100$ points for $x$ in $[0,1]$. Exact solution from~\cite{Benton1972}.}
    \label{fig:Burgers_sin}
\end{figure}

\section{Navier-Stokes equations for fluid dynamics}
\label{sec:application_NS}

\subsection{Model}

We now consider the 1D Navier-Stokes equations for fluid dynamics for which we have $p=3$, $\mathbf{u}^\varepsilon = [\rho, j, E]^T$ where $\rho$ is the density of mass, $j$ is the momentum and $E$ is the total energy by unit of mass. The convective flux is given by
\begin{align}
    f(\mathbf{u}^\varepsilon) = [j, j^2/\rho + P, (E+P)j/\rho]^T,
\end{align}
where $P$ is the thermodynamic pressure, related to $(\rho, E)$ by the ideal gas equation of state: $P = (\gamma-1) (E-j^2/(2\rho))$ and $\gamma$ is the heat capacity ratio of the gas. The diffusion matrix is given by:
\begin{align}
    \mathbf{D} = \nu
    \begin{bmatrix}
        0 & 0 & 0 \\
        -4/3 u & 4/3 & 0 \\
        -4/3 u^2 + \gamma/\mathrm{Pr}(u^2-E/\rho) & 4/3 u - \gamma u/\mathrm{Pr} & \gamma/\mathrm{Pr}
    \end{bmatrix},
\end{align}
where $u=j/\rho$ is the fluid velocity, $\nu=\mu/\rho$ is the kinematic viscosity, $\mu$ is the constant dynamic viscosity, $\mathrm{Pr}$ is the Prandtl number defined as
\begin{align}
    \mathrm{Pr} = \frac{\mu \gamma R}{\lambda (\gamma-1)},
\end{align}
$R$ is the gas constant and $\lambda$ is the thermal conductivity of the fluid. Note that this choice of $\mathbf{D}$ matrix is in line with the 1D projection of the 3D Navier-Stokes equations, for which a viscous stress tensor is defined as $\sigma = 4/3 \mu \partial_x u$. This matrix is diagonalizable with three non-negative eigenvalues that can be used to define a local Knudsen number: $(0, 4 \nu/3,  \gamma \nu/\mathrm{Pr})$. Also note that since there is no diffusion on the mass equation, $\mathbf{D}$ is not invertible. The use of a Lobato IIIC scheme as in section \ref{sec:RK} is therefore of paramount importance for this system of equations.

The two-wave model of Example~\ref{ex:Euler} is considered. The sub-characteristic condition is sufficient to make this model compatible with entropy inequalities. It reads
\begin{align}
    a > \max\limits_i(|u_i| + c_i),
\end{align}
where $c_i=\sqrt{\gamma P_i/\rho_i}$ is the sound speed and the index $i$ indicates here the discrete point in space. The inverse collision matrix is computed thanks to \eqref{Condition:T} and the Knudsen number is defined following \eqref{eq:def_epsilon} as
\begin{align}
	\label{eq:def_Knudsen_NS}
	\varepsilon = \frac{\mu}{a \ell \rho_c},
\end{align}
where $\ell$ is a characteristic length and $\rho_c$ a characteristic density. These parameters depend on the problem under consideration and will be provided for each of the test cases investigated below.

\subsection{Linear acoustics}

We first assess the ability of the model to deal with acoustic waves propagation in the linear approximation. To this extent, we assume that the solution of the Navier-Stokes equations has the form $\mathbf{u} (x, t) = \overline{\mathbf{u}} + \tilde{\mathbf{u}} (x,t)$, where $\overline{\mathbf{u}}$ is a mean base flow, constant in time and space, and $\tilde{\mathbf{u}}$ is a local perturbation of the flow. Assuming that $\tilde{\mathbf{u}} \ll \overline{\mathbf{u}}$, the Navier-Stokes equation can be linearized as
\begin{align}
    \frac{\partial \tilde{\mathbf{u}}}{\partial t} + \mathbf{f}'(\overline{\mathbf{u}}) \frac{\partial \tilde{\mathbf{u}}}{\partial x} = \mathbf{D}(\overline{\mathbf{u}}) \frac{\partial^2 \tilde{\mathbf{u}}}{\partial x^2}.
    \label{eq:linearized_NS}
\end{align}
We then assume that the perturbations are complex plane monochromatic waves: $\tilde{\mathbf{u}} = \hat{\mathbf{u}} \exp(\mathrm{i} (kx-\omega t))$, where $\hat{\mathbf{u}}$ is the complex amplitude of the wave, $k \in \mathbb{R}$ its wavenumber and $\omega \in \mathbb{C}$ its complex pulsation. Injecting this perturbation in Eq.~(\ref{eq:linearized_NS}) leads to the following eigenvalue problem:
\begin{align}
    \omega \tilde{\mathbf{u}} = \left[ k \mathbf{f}'(\overline{\mathbf{u}}) - \mathrm{i} k^2 \mathbf{D}(\overline{\mathbf{u}}) \right] \tilde{\mathbf{u}}.
    \label{eq:NS_eigenvalue_problem}
\end{align}
Solving this eigenvalue problem leads to the knowledge of eigenvectors of the flow $\hat{\mathbf{u}}$ and corresponding complex eigenvalues $\omega$ whose real part (resp. imaginary part) characterizes the propagation (resp. the temporal amplification) of the wave.

In the present study, we set $\overline{\mathbf{u}} = [\overline{\rho}, \overline{\rho} \overline{u}, \overline{P}/(\gamma-1) + \overline{\rho} \overline{u}^2/2]^T$ with $\overline{\rho} = 1$, $\overline{P} = 1$ and $\overline{u} = 2 \overline{c} = 2 \sqrt{\gamma}$ with $\gamma = 1.4$, in order to assess the ability of the model to simulate supersonic flows. Other parameters are: $\mu = 0.001$, $\mathrm{Pr} = 0.71$ and $k=2\pi$. A $L=1$-length 1D domain with periodic boundary conditions is initialized as follows: the eigenvalue problem of Eq.~(\ref{eq:NS_eigenvalue_problem}) is solved in order to retain the eigenvalue $\omega$ whose real part is the closest to $\overline{u}+\overline{c}$. By this procedure, a downstream acoustic wave can be isolated. The corresponding eigenvector $\hat{\mathbf{u}}$ is normalized such that $\phi(\hat{\rho})=0$, where $\phi(\hat{\rho})$ is the phase of the complex number $\hat{\rho}$ and $|\hat{\rho}| = 0.00001$ to satisfy the linear approximation, and the domain is initialized as
\begin{align}
    \mathbf{u}(x,0) = \overline{\mathbf{u}} + |\hat{\mathbf{u}}| \cos( kx + \phi(\hat{\mathbf{u}})).
\end{align} 
The numerical solution is to be compared with the exact one in the linear approximation:
\begin{align}
    \mathbf{u}_{exact}(x,t) = \overline{\mathbf{u}} + |\hat{\mathbf{u}}| \cos( kx - \mathrm{Re}(\omega) t + \phi(\hat{\mathbf{u}})) e^{\mathrm{Im}(\omega) t}.
\end{align}
For this case, the Knudsen number is defined using \eqref{eq:def_Knudsen_NS} with $\rho_c = \overline{\rho} = 1$ and $\ell = 1/k$.
Fig.~\ref{fig:Order_Acoustic_NS} displays the mesh convergence of the $L^2$ error in the density field measured at time $t=0.005$ in three cases: 
\begin{itemize}
	\item[(a)] $\mu = 0.001$, $a=1.1 \max(|u|+c)$ (corresponding to $\varepsilon \approx 0.0016$),
	\item[(b)] $\mu = 0.001$, $a=10 \max(|u|+c)$ (corresponding to $\varepsilon \approx 0.00018$),
	\item[(c)] $\mu=0$, $a=1.1 \max(|u|+c)$ (corresponding to $\varepsilon = 0$).
\end{itemize}
Similar conclusion can be drawn as in the advection-diffusion test case. Notably, the consistency error exhibited in case (a) is reduced by decreasing the Knudsen number as done in case (b). Furthermore, a decrease of $\varepsilon$ also leads to an increase in the $L^2$ error of the first-order scheme, and to a hyper-convergence of the second-order scheme. Also note that a consistency error remains in the Euler case (c). This is likely to be due to the linear assumption which is no more valid at these scales. Quantitative results of this convergence study are provided in Table~\ref{tab:Order_Acoustic_NS}.

\begin{figure}[h!]
    \centering
    \begin{subfigure}{0.32\textwidth}
    \centering
    \includegraphics{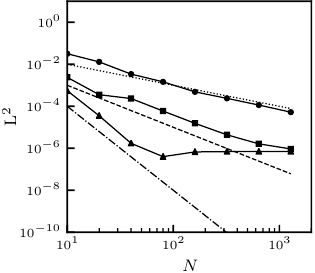}
    \caption{$\mu=0.001$, $a=1.1\ \max(|u|+c)$}
    \end{subfigure}
    \begin{subfigure}{0.32\textwidth}
    \centering
    \includegraphics{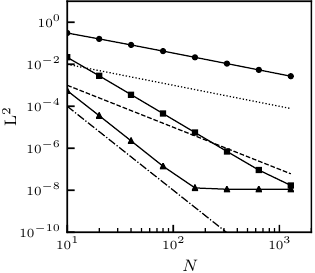}
    \caption{$\mu=0.001$, $a=10\ \max(|u|+c)$}
    \end{subfigure}
    \begin{subfigure}{0.32\textwidth}
    \centering
    \includegraphics{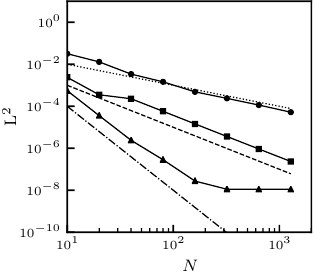}
    \caption{$\mu=0$, $a=1.1\ \max(|u|+c)$}
    \end{subfigure}
    \caption{Mesh convergence study of the acoustic propagation test case with the Navier-Stokes model at time $t=0.005$. Legend is similar as in Fig.~\ref{fig:Order_ADE}.}
    \label{fig:Order_Acoustic_NS}
\end{figure}

\begin{table}[!ht]
\begin{subtable}{1.\textwidth}
\centering
\begin{tabular}{|p{0.05\textwidth}||>{\centering}p{0.16\textwidth}>{\centering}p{0.04\textwidth}|>{\centering}p{0.16\textwidth}>{\centering}p{0.04\textwidth}|>{\centering}p{0.16\textwidth}>{\centering\arraybackslash}p{0.04\textwidth}|}
\hline
 & \multicolumn{2}{|c|}{First-order} & \multicolumn{2}{c|}{Second-order} & \multicolumn{2}{c|}{Fourth-order} \\
\hline
$h$ & $L^2$ & $r$ & $L^2$ & $r$ & $L^2$ & $r$ \\
\hline
$10$ & $3.18358494\ 10^{-2}$ & - & $2.38762319\ 10^{-3}$ & - & $5.52173687\ 10^{-4}$ & - \\ 
$20$ & $1.29954354\ 10^{-2}$ & $1.29$ & $3.56024569\ 10^{-4}$ & $2.75$ & $3.52472725\ 10^{-5}$ & $3.97$ \\ 
$40$ & $3.41042646\ 10^{-3}$ & $1.93$ & $2.33950628\ 10^{-4}$ & $0.61$ & $1.72009705\ 10^{-6}$ & $4.36$ \\ 
$80$ & $1.46187956\ 10^{-3}$ & $1.22$ & $5.94565368\ 10^{-5}$ & $1.98$ & $4.00382425\ 10^{-7}$ & $2.10$ \\ 
$160$ & $4.86899419\ 10^{-4}$ & $1.59$ & $1.54170403\ 10^{-5}$ & $1.95$ & $6.69901954\ 10^{-7}$ & $0.74$ \\ 
$320$ & $2.38695684\ 10^{-4}$ & $1.03$ & $4.38835080\ 10^{-6}$ & $1.81$ & $6.94795015\ 10^{-7}$ & $0.05$ \\ 
$640$ & $1.14526343\ 10^{-4}$ & $1.06$ & $1.62700122\ 10^{-6}$ & $1.43$ & $6.96842556\ 10^{-7}$ & $0.00$ \\ 
$1280$ & $5.24287832\ 10^{-5}$ & $1.13$ & $9.32826480\ 10^{-7}$ & $0.80$ & $6.96971908\ 10^{-7}$ & $0.00$ \\ 
\hline
\end{tabular}
\caption{$\mu=0.001$, $a=1.1\ \max(|u|+c)$ ($\varepsilon \approx 0.0016$)}
\end{subtable} \\ \vspace{1mm} \\
\begin{subtable}{1.\textwidth}
\centering
\begin{tabular}{|p{0.05\textwidth}||>{\centering}p{0.16\textwidth}>{\centering}p{0.04\textwidth}|>{\centering}p{0.16\textwidth}>{\centering}p{0.04\textwidth}|>{\centering}p{0.16\textwidth}>{\centering\arraybackslash}p{0.04\textwidth}|}
\hline
 & \multicolumn{2}{|c|}{First-order} & \multicolumn{2}{c|}{Second-order} & \multicolumn{2}{c|}{Fourth-order} \\
\hline
$h$ & $L^2$ & $r$ & $L^2$ & $r$ & $L^2$ & $r$ \\
\hline
$10$ & $3.08147409\ 10^{-1}$ & - & $2.13753486\ 10^{-2}$ & - & $5.52863168\ 10^{-4}$ & - \\ 
$20$ & $1.61484438\ 10^{-1}$ & $0.93$ & $2.83112661\ 10^{-3}$ & $2.92$ & $3.57950633\ 10^{-5}$ & $3.95$ \\ 
$40$ & $8.33846856\ 10^{-2}$ & $0.95$ & $3.58675606\ 10^{-4}$ & $2.98$ & $2.25533475\ 10^{-6}$ & $3.99$ \\ 
$80$ & $4.24743215\ 10^{-2}$ & $0.97$ & $4.49949168\ 10^{-5}$ & $2.99$ & $1.39777272\ 10^{-7}$ & $4.01$ \\ 
$160$ & $2.14510843\ 10^{-2}$ & $0.99$ & $5.63735081\ 10^{-6}$ & $3.00$ & $1.29570311\ 10^{-8}$ & $3.43$ \\ 
$320$ & $1.07810600\ 10^{-2}$ & $0.99$ & $7.09224153\ 10^{-7}$ & $2.99$ & $1.11811543\ 10^{-8}$ & $0.21$ \\ 
$640$ & $5.40483745\ 10^{-3}$ & $1.00$ & $9.16743840\ 10^{-8}$ & $2.95$ & $1.12763310\ 10^{-8}$ & $0.01$ \\ 
$1280$ & $2.70600638\ 10^{-3}$ & $1.00$ & $1.71030799\ 10^{-8}$ & $2.42$ & $1.12809781\ 10^{-8}$ & $0.00$ \\ 
\hline
\end{tabular}
\caption{$\mu=0.001$, $a=10\ \max(|u|+c)$ ($\varepsilon \approx 0.00018$)}
\end{subtable} \\ \vspace{1mm} \\
\begin{subtable}{1.\textwidth}
\centering
\begin{tabular}{|p{0.05\textwidth}||>{\centering}p{0.16\textwidth}>{\centering}p{0.04\textwidth}|>{\centering}p{0.16\textwidth}>{\centering}p{0.04\textwidth}|>{\centering}p{0.16\textwidth}>{\centering\arraybackslash}p{0.04\textwidth}|}
\hline
 & \multicolumn{2}{|c|}{First-order} & \multicolumn{2}{c|}{Second-order} & \multicolumn{2}{c|}{Fourth-order} \\
\hline
$h$ & $L^2$ & $r$ & $L^2$ & $r$ & $L^2$ & $r$ \\
\hline 
$10$ & $3.18450250\ 10^{-2}$ & - & $2.38857430\ 10^{-3}$ & - & $5.52869492\ 10^{-4}$ & - \\ 
$20$ & $1.29973579\ 10^{-2}$ & $1.29$ & $3.52777872\ 10^{-4}$ & $2.76$ & $3.59285924\ 10^{-5}$ & $3.94$ \\ 
$40$ & $3.41064593\ 10^{-3}$ & $1.93$ & $2.28925579\ 10^{-4}$ & $0.62$ & $2.39953322\ 10^{-6}$ & $3.90$ \\ 
$80$ & $1.46214532\ 10^{-3}$ & $1.22$ & $5.77456106\ 10^{-5}$ & $1.99$ & $2.84941443\ 10^{-7}$ & $3.07$ \\ 
$160$ & $4.86913078\ 10^{-4}$ & $1.59$ & $1.45129893\ 10^{-5}$ & $1.99$ & $2.74795789\ 10^{-8}$ & $3.37$ \\ 
$320$ & $2.38642326\ 10^{-4}$ & $1.03$ & $3.65387033\ 10^{-6}$ & $1.99$ & $1.12293677\ 10^{-8}$ & $1.29$ \\ 
$640$ & $1.14456539\ 10^{-4}$ & $1.06$ & $9.24355743\ 10^{-7}$ & $1.98$ & $1.10562078\ 10^{-8}$ & $0.02$ \\ 
$1280$ & $5.23548707\ 10^{-5}$ & $1.13$ & $2.35439111\ 10^{-7}$ & $1.97$ & $1.10582163\ 10^{-8}$ & $0.00$ \\ 
\hline
\end{tabular}
\caption{$\mu=0$, $a=1.1\ \max(|u|+c)$ ($\varepsilon =0$)}
\end{subtable}
\caption{Quantitative results of the $L^2$ errors shown in Fig.~\ref{fig:Order_Acoustic_NS}.}
\label{tab:Order_Acoustic_NS}
\end{table}

Similarly to what is proposed in Sec.~\ref{sec:application_scalars}, an asymptotic study in Knudsen number is then performed on a fine mesh of $N=1000$ points with the fourth-order model in order to get rid of numerical errors. The dynamic viscosity is set to $\mu = 0.1$ so that consistency errors in $\mathcal{O}(\varepsilon^2)$ are expected to be much larger than errors attributed to the linear approximation. Results shown in Table~\ref{tab:Order_Kn_NS} exhibits an effective second-order slope in $\varepsilon$.

\begin{table}[!ht]
\centering
\begin{tabular}{|p{0.13\textwidth}|>{\centering}p{0.11\textwidth}>{\centering}p{0.11\textwidth}>{\centering}p{0.11\textwidth}>{\centering}p{0.11\textwidth}>{\centering}p{0.11\textwidth}>{\centering\arraybackslash}p{0.135\textwidth}|}
\hline
$a/ \max(|u|+c)$ & $1.1$ & $2.2$ & $4.4$ & $8.8$ & $17.6$ & $35.2$ \\
\hline
$\varepsilon$ & $0.16$ & $0.08$ & $0.04$ & $0.02$ & $0.01$ & $0.005$ \\
\hline 
$L^2$ & $4.6585\ 10^{-4}$ & $2.8028\ 10^{-4}$ & $9.7336\ 10^{-5}$ & $2.5826\ 10^{-5}$ & $6.5393\ 10^{-6}$ & $1.6399\ 10^{-6}$ \\
$r$ & - & $0.73$ & $1.53$ & $1.91$ & $1.98$ & $2.00$ \\
\hline
\end{tabular}
\caption{Asymptotic study of the consistency error in Knudsen number $\varepsilon$ of an acoustic wave with the Navier-Stokes model. Simulations are performed with $\mu=0.1$ up to time $t=0.005$. In order to get rid of numerical errors, a fine mesh of 1000 points is considered and simulations are performed with the fourth-order scheme.}
\label{tab:Order_Kn_NS}
\end{table}

\subsection{Viscous steady shock}

We consider a steady viscous shock whose left and right state obey the following Rankine-Hugoniot relations:
\begin{align}
    (\rho, u, P)_L = (1, \mathrm{Ma}\sqrt{\gamma}, 1), \qquad (\rho, u, P)_R = \left (1/\theta, \theta \mathrm{Ma} \sqrt{\gamma}, \frac{\gamma+1 - \theta(\gamma-1)}{\theta(\gamma+1)-(\gamma-1)} \right),
\end{align}
where $\gamma=1.4$, $\mathrm{Ma}$ is the Mach number upstream of the shock and 
\begin{align}
    \theta = \frac{\gamma-1}{\gamma +1} + \frac{2}{(\gamma+1)\mathrm{Ma}^2}.
\end{align}
\clearpage

In the particular case $\mathrm{Pr}=3/4$, the 1D Navier-Stokes equations can be analytically solved to obtain an exact solution of the viscous shock profile~\cite{Zeldovich1967}. The latter reads
\begin{align}
    x = -\frac{8 \sqrt{\gamma} \mu}{3(\gamma+1)\mathrm{Ma}} \left[ \frac{\theta}{1-\theta} \log \left( \frac{v-\theta}{u_{in}-\theta} \right) - \frac{1}{1-\theta} \log \left( \frac{1-v}{1-u_{in}} \right) \right],
    \label{eq:viscous_shock_exact}
\end{align}
where $v=1/\rho$ and $u_{in}=(1+\theta)/2$ is the velocity at $x=0$. In the following, we set $\mu = 0.001$. Inverting Eq.~(\ref{eq:viscous_shock_exact}) allows us to compute the density profile, from which pressure, velocity and entropy $s$ can be computed as
\begin{align}
    & p = \frac{1}{v} \left( 1 + \frac{\gamma-1}{2} \mathrm{Ma}^2 (1-v^2) \right), \\
    & u = v \mathrm{Ma} \sqrt{\gamma}, \\
    & \eta=\eta_0 \log (p/\rho^\gamma),
\end{align}
where $\eta_0 = 1/(\gamma-1)$. A characteristic length related to the shock width can be defined as~\cite{Zeldovich1967}
\begin{align}
    \delta = \frac{2 \mathrm{Ma}}{\mathrm{Ma}^2-1} \mu \sqrt{\pi/2},
\end{align}
and, following \eqref{eq:def_Knudsen_NS}, the Knudsen number is defined as
\begin{align}
    \varepsilon = \frac{\mu}{a \delta},
\end{align}
where the density of the left state ($\rho=1$) has been considered as characteristic density $\rho_c$. A one-dimensional domain is initialized with
\begin{align}
    (\rho, u, P)(x,0) = \frac{1}{2}\left[ (\rho, u, P)_L + (\rho, u, P)_R \right] + \frac{1}{2} \left[ (\rho, u, P)_R - (\rho, u, P)_L \right] \tanh (x/(2\delta)).
\end{align}
The mesh size is $\Delta x = \delta/10$ and the length of the computational domain is $L=250 \delta$, large enough so that interactions with the boundary conditions (here imposed as Dirichlet boundaries) can be neglected when time convergence is reached.

Fig.~\ref{fig:ViscousShock_M2} displays the entropy profiles obtained for $\mathrm{Ma}=2$ in two cases: (a) $a=1.1 \max(|u|+c)$ and (b) $a=10 \max(|u|+c)$. They respectively correspond to $\varepsilon \approx 0.16$ and $\varepsilon \approx 0.017$. The profiles obtained with the second- and fourth-order schemes in Fig.~\ref{fig:ViscousShock_M2_a1.1} are in good agreement with the exact solution, except left of the peak where a slight overestimation of the entropy is obtained. This can be attributed to the consistency error, which is supported by Fig.~\ref{fig:ViscousShock_M2_a10} where a better agreement is obtained after reducing the Knudsen number.

\begin{figure}[h!]
    \centering
    \begin{subfigure}{0.47\textwidth}
    \includegraphics{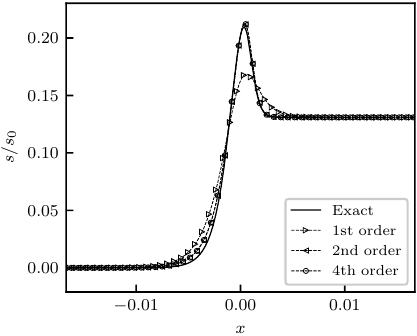}
    \caption{$a=1.1 \max(|u|+c)$ ($\varepsilon \approx 0.16$) \label{fig:ViscousShock_M2_a1.1}}
    \end{subfigure}
    \begin{subfigure}{0.47\textwidth}
    \includegraphics{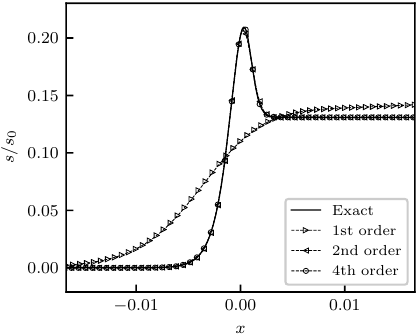}
    \caption{$a=10 \max(|u|+c)$ ($\varepsilon \approx 0.017$) \label{fig:ViscousShock_M2_a10}}
    \end{subfigure}
    \caption{Viscous shock with the Navier-Stokes model, $\mu=0.001$, $\mathrm{Pr} = 3/4$, $\gamma=1.4$. The Mach number is $\mathrm{Ma} = 2$.}
    \label{fig:ViscousShock_M2}
\end{figure}

A similar simulation performed at $\mathrm{Ma}=10$ is displayed in Fig.~\ref{fig:ViscousShock_M10} to illustrate the robustness and accuracy of the method for high Mach number flows. We can see that the consistency error observed in Fig.~\ref{fig:ViscousShock_M10_a1.1} is larger than in Fig.~\ref{fig:ViscousShock_M2_a1.1}, which can be attributed to a larger Knudsen number at this high Mach number. Still increasing $a$ to $10 \max(|u|+c)$ allows reducing the consistency error and leads to a very good agreement of the second- and fourth-order methods with the Navier-Stokes solution. 

\begin{figure}[h!]
    \centering
    \begin{subfigure}{0.47\textwidth}
    \includegraphics{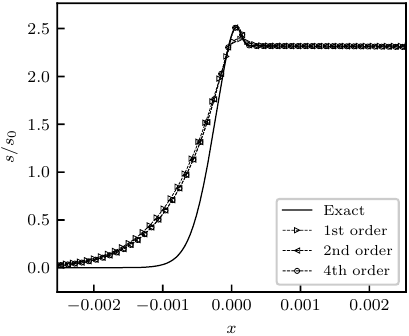}
    \caption{$a=1.1 \max(|u|+c)$ ($\varepsilon \approx 0.28$) \label{fig:ViscousShock_M10_a1.1} }
    \end{subfigure}
    \begin{subfigure}{0.47\textwidth}
    \includegraphics{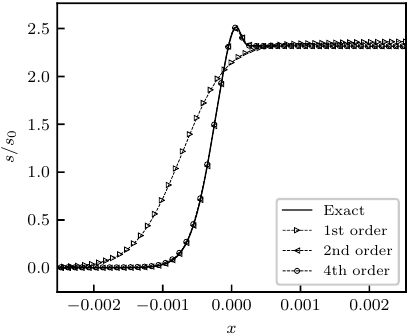}
    \caption{$a=10 \max(|u|+c)$ ($\varepsilon \approx 0.03$)}
    \end{subfigure}
    \caption{Viscous shock with the Navier-Stokes model, $\mu=0.001$, $\mathrm{Pr} = 3/4$, $\gamma=1.4$. The Mach number is $\mathrm{Ma} = 10$. \label{fig:ViscousShock_M10_a10}}
    \label{fig:ViscousShock_M10}
\end{figure}

A final discussion can be held regarding the validity of the Navier-Stokes solution for such a simulation. It is well known that the Navier-Stokes equations are no more valid for simulating hypersonic flows, for which large off-equilibrium phenomena have to be considered. In fact, it is even not valid to correctly simulate the shock width in the case $\mathrm{Ma}=2$. This is due to the fact that the Navier-Stokes equations are only valid as long as the representative length scale of the problem is much larger than the mean free path of the particles. This assumption is commonly referred to as the continuum assumption. However, the characteristic width of a shock is precisely in the order of magnitude of the mean free path. Hence, the Navier-Stokes equations themselves may not be valid for the viscous shock simulations performed in this section, especially in the case $\mathrm{Ma}=10$, so that the consistency error obtained with the kinetic models may not be so problematic. To be specific, regarding Fig.~\ref{fig:ViscousShock_M2_a1.1} and Fig.~\ref{fig:ViscousShock_M10_a1.1}, it is not sure that the exact Navier-Stokes solution is more representative of the physics than the one obtained by the kinetic model:  they both share a $\mathcal{O}(\varepsilon^2)$ error with the kinetic theory of gases.

\section{Conclusion}
We have presented a framework that allows us to approximate the solution of convection-diffusion like problems using a kinetic approach. Linear and non-linear examples are considered and discussed, including the Navier-Stokes equations. The strategy adopted here considerably differs from previous work, where the convection-diffusion PDE is recovered in the limit of a relaxation parameter $\varepsilon \rightarrow 0$, and where kinetic velocities scaling as $\mathcal{O}(1/\varepsilon)$ are often to be considered. In the present work, we do not look at the formal limit $\varepsilon \rightarrow 0$, but perform an asymptotic expansion for small values of $\varepsilon$ in order to match the diffusive flux of the PDE at first-order in $\varepsilon$. This framework, very different from the previous work, is motivated by the kinetic theory of gases, where the NS equations are not a limit of the BGK equation but a correction of the Euler equations at first-order in the Knudsen number. This approach notably requires a proper definition of the Knudsen number on a case by case basis, to measure how the relaxation parameter can be reasonably considered small. The price to pay is that the expected PDE is recovered up to a consistency error scaling as $\mathcal{O}(\varepsilon^2)$.

Once the model is set up, we discuss in length how to discretize it with arbitrary order, in time and space. First-, second-, and fourth-order methods are provided, and the expected orders of accuracy are recovered until the consistency error. Interestingly, we show how the latter can be arbitrarily reduced by increasing the velocity norm of the kinetic model, which is a free parameter as far as the subcharacteristic condition is satisfied. In this regard, the method we propose may seem not so different from previous work: the consistency error vanishes, i.e. the PDE is \textit{exactly} solved, in the limit of infinitely large kinetic velocities. The key point is to accept the existence of the consistency error and to control it in order to build methods that are able to approximate a given linear or non-linear partial differential equation with a given accuracy.

So far the method is described for one dimensional problems. The extension to several dimensions is in progress and will be the topic of a future publication.
\section*{Acknowledgements}
Lorenzo Micalizzi is gratefully acknowledged for fruitful discussions regarding DeC methods. GW has been funded by SNFS grants \# 
200020\_204917 ``Structure preserving and fast methods for hyperbolic systems of conservation laws'' and 
FZEB-0-166980.

\bibliographystyle{unsrt}
\bibliography{main6}
\end{document}